\documentclass[a4paper]{article}
\usepackage{latexsym}
\usepackage{amssymb}
\usepackage{graphicx}
\usepackage{color}
\usepackage{amsmath}
\usepackage{float}
\usepackage{mathrsfs}
\usepackage[latin1]{inputenc}
\usepackage{amsthm}
\usepackage{verbatim}
\usepackage{url}
\usepackage{vmargin}

\setmarginsrb{20mm}{20mm}{20mm}{20mm}{10mm}{10mm}{10mm}{10mm}

\RequirePackage{amsthm}
\newtheorem{theorem}{Theorem}
\newtheorem{lemma}[theorem]{Lemma}
\newtheorem{cor}[theorem]{Corollary}
\theoremstyle{definition}
\newtheorem{defn}[theorem]{Definition}
\theoremstyle{remark}
\newtheorem{rem}[theorem]{Remark}
\DeclareMathOperator{\diam}{diam}

\def\esssup{\operatornamewithlimits{ess\,sup}}
\def\essinf{\operatornamewithlimits{ess\,inf}}

\def\PP{\mathcal{P}}
\def\PPln{\mathcal{P}^{\log}}

\newcommand{\R}{{\mathbb R}}

\newcommand{\Rn}{{\mathbb R}^n}

\newcommand{\Om}{\Omega}

\newcommand{\ti}{\tilde }

\newcommand{\eps}{\varepsilon}

\newcommand{\bdry}{\partial}

\def\div{{\rm div}}
\def\capp{{\rm cap}}
\def\px{{p(\cdot)}}
\def\p{p(x)}

\newcommand{\kom}[1]{}
\renewcommand{\kom}[1]{{\bf [#1]}}

\definecolor{blau}{rgb}{0.1,0.0,0.9}

\newcounter{komcounter}
\numberwithin{komcounter}{section}

\def\vint{\mathop{\mathchoice%
          {\setbox0\hbox{$\displaystyle\intop$}\kern 0.22\wd0%
           \vcenter{\hrule width 0.6\wd0}\kern -0.82\wd0}%
          {\setbox0\hbox{$\textstyle\intop$}\kern 0.2\wd0%
           \vcenter{\hrule width 0.6\wd0}\kern -0.8\wd0}%
          {\setbox0\hbox{$\scriptstyle\intop$}\kern 0.2\wd0%
           \vcenter{\hrule width 0.6\wd0}\kern -0.8\wd0}%
          {\setbox0\hbox{$\scriptscriptstyle\intop$}\kern 0.2\wd0%
           \vcenter{\hrule width 0.6\wd0}\kern -0.8\wd0}}%
          \mathopen{}\int}

\begin{document}

\title{The boundary Harnack inequality for variable exponent $p$-Laplacian, Carleson estimates, barrier functions and $\px$-harmonic measures}

\author{
Tomasz Adamowicz{\small{$^1$}}
\\
\it\small Institute of Mathematics of the Polish Academy of Sciences \\
\it\small 00-956 Warsaw, Poland\/{\rm ;}
\it\small T.Adamowicz@impan.pl
\\
\\
Niklas L.P. Lundstr\"{o}m
\\
\it \small Department of Mathematics and Mathematical Statistics, Ume{\aa} University\\
\it \small SE-90187 Ume{\aa}, Sweden\/{\rm ;}
\it \small niklas.lundstrom@math.umu.se
}

\date{}
\maketitle

\footnotetext[1]{T. Adamowicz was supported by a grant of National Science Center, Poland (NCN),
UMO-2013/09/D/ST1/03681.}

\begin{abstract}

We investigate various boundary decay estimates for $\px$-harmonic functions.
For domains in $\R^n, n\geq 2$ satisfying the ball condition ($C^{1,1}$-domains) we show the boundary Harnack inequality for $\px$-harmonic functions under the assumption that the variable exponent $p$ is a bounded Lipschitz function.
The proof involves barrier functions and chaining arguments.
Moreover, we prove a Carleson type estimate for $\px$-harmonic functions in NTA domains in $\R^n$ and provide lower- and upper- growth estimates and a doubling property for a $\px$-harmonic measure.\\

\medskip


\end{abstract}

\setcounter{theorem}{0}
\setcounter{equation}{0}


\section{Introduction}

\setcounter{theorem}{0}
\setcounter{equation}{0}

 The studies of boundary Harnack inequalities for solutions of differential equations have a long history.
In the setting of harmonic functions on Lipschitz domains such a result was first proposed by Kemper~\cite{K} and later studied by Ancona~\cite{An}, Dahlberg~\cite{D77} and Wu~\cite{W}. Subsequently, Kemper's result was extended by Caffarelli-Fabes-Mortola\--Salsa \cite{CFMS} to a class of elliptic equations, by Jerison--Kenig \cite{JK82} to the setting of non-tangentially accessible (NTA) domains, Ba\~nuelos--Bass--Burdzy \cite{bbb} and Bass--Burdzy \cite{bburd} studied the case of H\"older domains while Aikawa \cite{Ai} the case of uniform domains. The extension of these results to the more general setting of $p$-harmonic operators turned out to be difficult, largely due to the nonlinearity of $p$-harmonic functions for $p\not=2$.
However, recently there has been a substantial progress in studies of boundary Harnack inequalities for nonlinear Laplacians:  Aikawa--Kilpel\"ainen--Shanmugalingam--Zhong~\cite{AKSZ07} studied the case of $p$-harmonic functions in $C^{1,1}$-domains, while Lewis--Nystr\"om~\cite{LN07, LN10, LN12} considered more general geometry such as
Lipschitz and Reifenberg-flat domains. Lewis--Nystr\"om results have been partially generalized to operators with variable coefficients, Avelin--Lundstr\"om--Nystr\"om~\cite{lower_order}, Avelin--Nystr\"om~\cite{avelin_kaj},
and to $p$-harmonic functions in the Heisenberg group, Nystr\"om~\cite{Kaj_new}.
Moreover, in \cite{N11} the second author proved a boundary Harnack inequality for $p$-harmonic
functions with $n<p\leq \infty$ vanishing on a $m$-dimensional hyperplane in $\mathbb{R}^n$ for $0 \leq m \leq n-1$.
We also refer to Bhattacharya~\cite{B02} and Lundstr\"om--Nystr\"om~\cite{LN11} for the case $p = \infty$,
where the latter investigated $A$-harmonic and Aronsson-type equations in planar uniform domains.
Concerning applications of boundary Harnack inequalities we mention free boundary problems and
studies of the Martin boundary.

 Another recently developing branch of nonlinear analysis is the area
of differential equations with nonstandard growth (variable exponent analysis)
and related variational functionals. The following equation, called the $\px$-Laplace equation,
serves as the model example:
\begin{equation}\label{eq:intro}
\div(|\nabla u|^{\px-2}\nabla u)\,=\,0,
\end{equation}
for a measurable function $p:\Om\to [1,\infty]$ called a variable exponent.
The variational origin of this equation naturally implies that solutions belong to the appropriate Musielak-Orlicz space $W^{1,\px}(\Om)$ (see
Preliminaries). If $p=const$, then this equation becomes the classical $p$-Laplacian.

 Apart from interesting theoretical considerations such equations arise in the applied sciences, for instance
in fluid dynamics, see e.g. Diening--R\r u\v zi\v cka~\cite{dr}, in the study of image processing, see for example Chen-Levine-Rao~\cite{clr} and electro-rheological fluids, see e.g. Acerbi--Mingione~\cite{AM, am02}; we also refer to Harjulehto--H\"ast\"o--L\^e--Nuortio~\cite{hhn} for a recent survey and further
references. In spite of the symbolic similarity to the constant exponent $p$-harmonic equation, various unexpected phenomena may occur when the exponent is a function, for instance the minimum of the
$\px$-Dirichlet energy may not exist even in the one-dimensional
case for smooth functions $p$; also smooth functions need not be
dense in the corresponding variable exponent Sobolev spaces. Although equation \eqref{eq:intro} is the Euler-Lagrange equation of the $\px$-Dirichlet energy, and thus is natural to study, it has many disadvantages comparing to the $p=const$ case. For instance solutions of \eqref{eq:intro} are, in general, not scalable, also the Harnack inequality is nonhomogeneous with constant depending on solution. In a consequence, the analysis of nonstandard growth equation is often difficult and leads to technical and nontrivial estimates (nevertheless, see Adamowicz--H\"ast\"o~\cite{adha1, adha2} for a variant of equation~\eqref{eq:intro} that overcomes some of the aforementioned difficulties, the so-called strong $\px$-harmonic equation).

\medskip

 The main goal of this paper is to show the boundary Harnack inequality for $\px$-harmonic functions on domains satisfying the ball condition (see Theorem~\ref{thm:harnack} below). Let us briefly describe main ingredients leading to this result, as it requires number of auxiliary lemmas and observations which are interesting per se and can be applied in other studies of variable exponent PDEs.

 In Section~\ref{Sect-Carleson} we study oscillations of $\px$-harmonic functions close to the boundary of a domain and prove, among other results, variable exponent Carleson estimates on NTA-domains, cf. Theorem~\ref{Carleson}. Similar estimates play important role, for instance in studies of the Laplace operator, in particular in relations between the topological boundary and the Martin boundary of the given domain; also in the $p$-harmonic analysis (see presentation in Section~\ref{Sect-Carleson} for further details and references). The main tools used in the proof of Theorem~\ref{Carleson}
are Hölder continuity up to the boundary, Harnack's inequality and an argument by
Caffarelli--Fabes--Mortola--Salsa \cite{CFMS} which, in our situation,
relies on various geometric concepts such as quasihyperbolic geodesics and related chaining arguments;
also on characterizations of uniform and NTA domains.

Section~\ref{Sect-barriers} is devoted to introducing two types of barrier functions,
called Wolanski-type and Bauman-type barrier functions, respectively.
In the analysis of PDEs, barrier functions appear, for example, in comparison arguments and in establishing growth conditions for functions, see e.g. Aikawa--Kilpel\"ainen--Shanmugalingam--Zhong~\cite{AKSZ07}, Lundstr\"om~\cite{N11}, Lundstr\"om--Vasilis~\cite{LV13} for the setting of $p$-harmonic functions. Furthermore, barriers can be applied in the solvability of the Dirichlet problem, especially in studies of regular points, see e.g. Chapter~6 in Heinonen--Kilpeläinen--Martio~\cite{hkm} and Chapter~11 in Bj\"orn--Bj\"orn~\cite{bb}. We would like to mention that our results on barriers
enhance the existing results in variable exponent setting, see Remark~\ref{Wolanski-comment}.

In Section~\ref{Sect-estimates} we prove our main results, a boundary Harnack inequality and growth estimates
for $\px$-harmonic functions vanishing on a portion of the boundary of a domain
$\Omega \subset \mathbb{R}^n$ satisfying the ball condition.
We refer to Section \ref{sect:prelim} for a definition of the ball condition
and point out that a domain satisfies the ball condition if and only if its boundary is $C^{1,1}$-regular.
Let us now briefly sketch our results. Let $w \in \partial \Omega$ and $r>0$ be small and suppose that $p$ is a bounded Lipschitz continuous variable exponent.
Assume that $u$ is a positive $\px$-harmonic function in
$\Omega \cap B(w, r)$ vanishing continuously on $\partial \Omega \cap B(w,r)$.
Then we prove that
\begin{align}\label{eq:intro1}
\frac1C \frac{d(x, \bdry\Om)}{r} \leq u(x) \leq  C \frac{d(x, \bdry\Om)}{r} \qquad \text{whenever} \quad x
\in \Omega \cap B(w, r/\ti c),
\end{align}
 for constants $\ti c$ and $C$ whose values depend on the geometry of $\Om$, variable exponent $p$ and certain features of $u$ and $v$, see the statement of Theorem~\ref{thm:harnack}. Here $d(x, \bdry\Om)$ denotes the Euclidean distance from $x$ to $\partial\Omega$.
Inequality \eqref{eq:intro1} says that $u$ vanishes at the same rate as the distance to the boundary
when $x$ approaches the boundary.

Suppose that $v$ satisfies the same assumptions as $u$ above.
An immediate consequence of \eqref{eq:intro1} is then the following boundary Harnack inequality:
\begin{align*}
\frac1C \leq \frac{u(x)}{v(x)} \leq  C  \qquad \text{whenever} \quad x
\in \Omega \cap B(w, r/\ti c),
\end{align*}
saying that $u$ and $v$ vanishes at the same rate as $x$ approaches the boundary (see Theorem \ref{thm:harnack} in Section~\ref{Sect-estimates}).
Among main tools used in the proof of boundary Harnack estimates let us mention Lemmas~\ref{le:lower} and \ref{le:upper} where we show the lower- and upper estimates for the rate of decay of a $\px$-harmonic function close to a boundary of the domain. It turns out that the geometry of the domain affects the number and type of parameters on which the rate of decay depends. Namely, our estimates depend on whether a domain satisfies the interior ball condition or the ball condition in Lemma 5.1 , cf. parts (i) and (ii) of Lemma 5.1.
Besides the ball condition, the proof of \eqref{eq:intro1} uses the barrier functions derived in Section~\ref{Sect-barriers}, the comparison principle and Harnack's inequality.
Our approach extends arguments from Aikawa--Kilpel\"ainen--Shanmugalingam--Zhong~\cite{AKSZ07}
to the case of variable exponents. We point out that the constants in \eqref{eq:intro1}, and thus also in the boundary Harnack inequality, depend on $u$ and $v$. Such a dependence is expected for variable exponent PDEs and difficult to avoid, as e.g. parameters in the Harnack inequality Lemma~\ref{harnack} and the barrier functions depend on solutions as well.

 Finally, in Section~\ref{Sect-measures} we define and study a lower and upper estimates for a $\px$-harmonic measure. We also prove a weak doubling property for such measures.
In the constant exponent setting similar results were obtained by Eremenko--Lewis~\cite{EL91}, Kilpel\"ainen--Zhong~\cite{KZ03} and Bennewitz--Lewis \cite{BL}.
For $p = const$, $p$-harmonic measures were employed to prove boundary Harnack
inequalities, see e.g. \cite{BL}, Lewis--Nystr\"om~\cite{LN08} and Lundstr\"om--Nystr\"om~\cite{LN11}.
The $p$-harmonic measure, defined as in the aforementioned papers, as well as boundary Harnack inequalities, have played a significant role when studying free boundary problems, see e.g. Lewis--Nystr\"om \cite{LN12}.

\section{Preliminaries}\label{sect:prelim}

We let $ \bar \Om$ and $\partial \Om$ denote, respectively, the closure and the boundary of the
set $\Om \subset \R^{n}$, for $n\geq 2$. We define $d(y,
\Om)$ to equal the Euclidean distance from $y \in \R^{n} $ to $\Om$, while $ \langle \cdot,
\cdot \rangle$ denotes the standard inner product on $ \R^{2} $ and
$ | x | = \langle x, x \rangle^{1/2} $ is the Euclidean norm of $x$.
Furthermore, by $B(x, r)=\{y \in \R^{n} : |x - y|<r\}$ we denote a
ball centered at point $x$ with radius $r>0$ and we let $dx$ denote
the $n$-dimensional Lebesgue measure on $\R^{n}$. If $\Om \subset
\R^n $ is open and $1 \leq q < \infty$, then by $W^{1 ,q} (\Om)$,
$W^{1 ,q}_0 (\Om)$ we denote the standard Sobolev space and the Sobolev space of functions with zero boundary values, respectively. Moreover, let $\Delta(w, r) = B(w,r)\cap \partial \Omega$. By $f_A$ we denote the integral average of $f$ over a set $A$.

For background on variable exponent function spaces we refer to the
monograph by Diening--Harjulehto--H\"ast\"o--R\r u\v zi\v cka~\cite{DHHR}.

A measurable function $p\colon \Omega\to [1,\infty]$ is called a
\emph{variable exponent}, and we denote
\[
p^+_A:=\esssup_{x\in A} p(x),\quad p^-_A:=\essinf_{x\in A} p(x),
\quad p^+:=p_\Omega^+ \quad\text{and}\quad p^-:=p_\Omega^-
\]
for $A\subset\Omega$.  If $A=\Om$ or if the underlying domain is fixed, we will often skip the index and set $p_A=p_\Om=p$.

In this paper we assume that our variable exponent functions are bounded, i.e.
\[
 1<p^-\leq p(x)\leq p^+<\infty\qquad \hbox{for almost every } x\in \Om.
\]
The set of all such exponents in $\Om$ will be denoted  $\PP(\Om)$.

The function $\alpha$ defined in a bounded domain $\Omega$ is said to be
\textit{$\log$-H\"older continuous} if there is constant
$L>0$ such that
\[
|\alpha(x)-\alpha(y)|\leq \frac{L}{\log(e+1/|x-y|)}
\]
for all $x,y\in \Omega$. We denote $p\in \PPln(\Omega)$ if
$1/p$ is $\log$-H\"older continuous; the smallest constant for which $\frac1p$ is
$\log$-H\"older continuous is denoted by $c_{\log}(p)$. If $p\in \PPln(\Omega)$,
then
\begin{equation}\label{logH-ball}
|B|^{p^+_B} \approx |B|^{p^-_B} \approx |B|^{p(x)} \approx |B|^{p_B}
\end{equation}
for every ball $B\subset \Omega$ and $x\in B$; here $p_B$ is the harmonic
average, $\frac1{p_B} := \vint_B \frac1{p(x)}\, dx$. The constants in the equivalences depend on $c_{\log}(p)$ and $\diam \Omega$. One of the immediate consequences of (\ref{logH-ball}) is that if $x\in B(w,r)$, then
\begin{equation}\label{logH-ball2}
\frac1c r^{-p(w)}\leq r^{-p(x)}\leq c r^{-p(w)}
\end{equation}
with $c$ depending only on constants in (\ref{logH-ball}).

 In this paper we study only log-H\"older continuous or Lipschitz continuous variable exponents. Both types of exponents can be extended to the whole $\R^n$ with their constants unchanged, see \cite[Proposition 4.1.7]{DHHR} and McShane-type extension result in Heinonen~\cite[Theorem~6.2]{hei01}, respectively. Therefore, without loss of generality we assume below that variable exponents are defined in the whole $\R^n$.
\smallskip

We define a \emph{(semi)modular} on the set of measurable functions
by setting
\[
\varrho_{L^{p(\cdot)}(\Omega)}(u) :=\int_{\Omega} |u(x)|^{p(x)}\,dx;
\]
here we use the convention $t^\infty = \infty \chi_{(1,\infty]}(t)$
in order to get a left-continuous modular, see
\cite[Chapter~2]{DHHR} for details. The \emph{variable
exponent Lebesgue space $L^{p(\cdot)}(\Omega)$} consists of all
measurable functions $u\colon \Omega\to \R$ for which the modular
$\varrho_{L^{p(\cdot)}(\Omega)}(u/\mu)$ is finite for some $\mu >
0$. The Luxemburg norm on this space is defined as
\[
\|u\|_{L^{p(\cdot)}(\Omega)}:= \inf\Big\{\mu > 0\,\colon\,
\varrho_{L^{p(\cdot)}(\Omega)}\big(\tfrac{u}\mu\big)\leq 1\Big\}.
\]
Equipped with this norm, $L^{p(\cdot)}(\Omega)$ is a Banach space. The variable exponent Lebesgue space is a special case of an Orlicz-Musielak space. For a constant function $p$, it coincides
with the standard Lebesgue space. Often it is assumed that
$p$ is bounded, since this condition is known to imply many desirable
features for $L^{p(\cdot)}(\Omega)$.

There is not functional relationship between norm and modular, but we
do have the following useful inequality:
\begin{equation}
\min\Big\{ \varrho_{L^{p(\cdot)}(\Omega)}(f)^\frac1{p^-}, \varrho_{L^{p(\cdot)}(\Omega)}(f)^\frac1{p^+} \Big\}
\le
\| f\|_{L^{p(\cdot)}(\Omega)}
\le
\max\Big\{ \varrho_{L^{p(\cdot)}(\Omega)}(f)^\frac1{p^-}, \varrho_{L^{p(\cdot)}(\Omega)}(f)^\frac1{p^+} \Big\}.
\label{ineq:norm-mod}
\end{equation}

One of the consequences of these relations is the so-called \emph{unit ball property}:
\begin{equation}\label{unit-ball}
\varrho_{L^{p(\cdot)}(\Omega)}(f)\leq 1\, \Rightarrow\, \| f\|_{L^{p(\cdot)}(\Omega)}\leq 1\quad \hbox{ and }\quad
\varrho_{L^{p(\cdot)}(\Omega)}(f)^\frac1{p^-} \le \| f\|_{L^{p(\cdot)}(\Omega)} \le \varrho_{L^{p(\cdot)}(\Omega)}(f)^\frac1{p^+}.
\end{equation}

If $E$ is a measurable set of finite measure, and $p$ and $q$
are variable exponents satisfying $q\leq p$, then $L^{p(\cdot)}(E)$
embeds continuously into $L^{q(\cdot)}(E)$. In particular,
every function $u\in L^{p(\cdot)}(\Omega)$ also belongs to
$L^{p_\Omega^-}(\Omega)$. The variable exponent H\"older inequality
takes the form
\begin{equation}
\int_\Omega f g \,dx \le 2 \, \|f\|_{L^{p(\cdot)}(\Omega)} \|g\|_{L^{p'(\cdot)}(\Omega)},\label{ineq:Holder}
\end{equation}
where $p'$ is the point-wise \textit{conjugate exponent}, $1/p(x)
+1/p'(x)\equiv 1$.
\bigskip

The \emph{variable exponent Sobolev space $W^{1,p(\cdot)}(\Omega)$}
consists of functions $u\in L^{p(\cdot)}(\Omega)$ whose
distributional gradient $\nabla u$ belongs to
$L^{p(\cdot)}(\Omega)$. The variable exponent Sobolev space
$W^{1,p(\cdot)}(\Omega)$ is a Banach space with the norm
\begin{displaymath}
\|u\|_{L^{p(\cdot)}(\Omega)}+\|\nabla u\|_{L^{p(\cdot)}(\Omega)}.
\end{displaymath}
In general, smooth functions are not dense in the variable exponent
Sobolev space, see Zhikov~\cite{Zh06} but the log-H\"older condition suffices
to guarantee that they are, see Diening--Harjulehto--H\"ast\"o--R\r u\v zi\v cka~\cite[Section~8.1]{DHHR}.
In this case, we define \emph{the Sobolev space with zero
boundary values}, $W_0^{1,\p}(\Omega)$, as the closure of $C_0^\infty(\Omega)$
in $W^{1,\px}(\Omega)$.

The Sobolev conjugate exponent is also defined point-wise,
$p^*(x):= \frac{n p(x)}{n-p(x)}$ for $p^+<n$. If $p$ is $\log$-H\"older
continuous, the Sobolev--Poincar\'e inequality
\begin{equation}\label{var-exp-Sob-Poin}
\| u-u_\Omega\|_{L^{p^*(\cdot)}(\Omega)}\le c\, \| \nabla u\|_{L^\px(\Omega)}
\end{equation}
holds when $\Omega$ is a nice domain, for instance convex or John
\cite[Section~7.2]{DHHR}. If $u\in W^{1,\px}_0(\Omega)$, then
the inequality $\| u\|_{L^{p^*(\cdot)}(\Omega)}\le c\, \| \nabla u\|_{L^\px(\Omega)}$
in any open set $\Omega$.

\begin{defn}\label{C-px-defn}
 The \emph{Sobolev $\px$-capacity} of a set $\Om \subset\R^n$ is defined as
\[
 C_{\px}(\Om):=\inf_{u}\int_{\R^n} (|u|^{\p}+|\nabla u|^{\p})\, dx,
\]
where the infimum is taken over all
$u\in W^{1,\px}(\R^n)$ such that $u\geq 1$  in a neighbourhood of  $\Om$.
\end{defn}

The properties of $\px$-capacity are similar to those in the constant case, see Theorem~10.1.2 in~\cite{DHHR}.
In particular $C_{\px}$ is an outer measure, see Theorem~10.1.1 in~\cite{DHHR}.

Another type of capacity used in the paper is the so-called \emph{relative $\px$-capacity} which appears for instance in the context of uniform $\px$-fatness (see next section and Chapter 10.2 in~\cite{DHHR} for more details).

\begin{defn}\label{cap-px-defn}
The \emph{relative $\px$-capacity} of a compact set $K\subset \Om$
is a number defined by
\begin{equation*}
\capp_{\px}(K, \Om) = \inf_{u} \int_{\Om} |\nabla u|^{\p}dx,
\end{equation*}
where the infimum is taken over all $u\in C_{0}^{\infty}(\Om)\cap W^{1, \px}(\Om)$ such that $u\geq 1$ in K.
\end{defn}
The definition extends to the setting of general sets in $\R^n$ in the same way as in the case of constant $p$, cf. \cite{DHHR} for details and further properties of the relative $\px$-capacity. In what follows we will need the following estimate, see Proposition~10.2.10 in \cite{DHHR}: for a bounded log-H\"older continuous variable exponent
$p:B(x, 2r)\to (1,n)$ it holds that
\begin{equation}\label{cap-est}
 c(n, p)r^{n-\p}\leq \capp_{\px}(\overline{B}(x, r), B(x, 2r)).
\end{equation}
The similar upper estimate holds for $r\leq 1$, cf. Lemma~10.2.9 in \cite{DHHR}.

\begin{defn}\label{px-defn}
A function $u \in W^{1, \px}_{loc}(\Om)$ is a\/
\textup{(}sub\/\textup{)}solution if
\begin{equation} \label{eq-sub-sol}
       \int_\Om |\nabla u|^{\p-2} \nabla u \cdot \nabla \phi\, dx (\le)= 0
\end{equation}
for all\/ \textup{(}nonnegative\/\textup{)} $\phi \in C_0^\infty(\Om)$.
\end{defn}
In what follows we will exchangeably be using terms (sub)solution and $\px$-(sub)solution. Similarly, we say that $u$ is a \emph{supersolution ($\px$-supersolution)} if $-u$ is a subsolution. A function which is both a subsolution and a supersolution is called a (weak) solution to the $\px$-harmonic equation. A continuous weak solution is called a \emph{$\px$-harmonic function}.

Among properties of $\px$-harmonic functions let us mention that
they are locally $C^{1,\alpha}$, see e.g. Acerbi--Mingione~\cite{AM}
or Fan~\cite[Theorem 1.1]{Fan}. Another tool, crucial from our point
of view, is the comparison principle.
\begin{lemma}[cf. Lemma 3.5 in Harjulehto--H\"ast\"o--Koskenoja--Lukkari--Marola~\cite{HHKLM}]\label{comparison_princ}
Let $u$ be a supersolution and $v$ a subsolution such that $u>v$ on $\partial \Om$ in the Sobolev sense. Then $u>v$ a.e. in $\Om$.
\end{lemma}
\noindent
By the standard reasoning the comparison principle implies the following maximum principle: \emph{If $u\in W^{1,\px}(\Om)\cap C(\overline{\Om})$ is a $\px$-subsolution in $\Om$, then the maximum of $u$ is attained at the boundary of $\Om$}.
For further discussion on comparison principles in the variable exponent setting we refer e.g. to Section 3 in Adamowicz--Bj\"orn--Bj\"orn~\cite{ABB13}.

We close our discussion of basic definitions and results with a presentation of the geometric concepts used in the paper.
\begin{defn}\label{def:uniform}
A domain $\Om\subset \R^n$ is called a \emph{uniform domain} if there exists a
constant $M_\Om \ge 1$, called a \emph{uniform constant}, such that whenever $x,y\in\Om$ there is a
rectifiable curve $\gamma:[0,l(\gamma)]\to\Om$, parameterized by arc length, connecting $x$ to $y$ and
satisfying the following two conditions:
\[
      l(\gamma) \le M_\Om |x-y|,
\]
and
\[
    \min\{|x-z|,|y-z| \}\le M_\Om  d(z, \partial \Om)
    \quad \text{for each point}\, z\in \gamma.
\]
\end{defn}
\begin{defn}\label{def:NTA}
A uniform domain $\Om\subset \R^n$ with constant $M_{\Omega}$ is called a
\emph{non-tangentially accessible (NTA) domain} if $\Om$ and its complement $\R^n\setminus \Om$ satisfy, additionally, the so-called \emph{corkscrew condition}:

For some $r_{\Om} > 0$ and for any $w \in \partial{\Omega}$ and $r \in (0, r_\Om)$, there exists a point $a_r(w) \in \Omega$ such that
\begin{equation*}
\frac{r}M_{\Om}< | a_r(w) - w| < r \qquad \mbox{and} \qquad d\bigl(a_r(w), \partial{\Omega}\bigr) > \frac{r}M_{\Om}.
\end{equation*}
\end{defn}
We note that in fact the (interior) corkscrew condition is implied by a uniform domain,
see Bennewitz--Lewis \cite{BL} and Gehring \cite{G87}.
Among examples of NTA domains we mention quasidisks,
bounded Lipschitz domains and domains with fractal boundary such as the von Koch snowflake.
A domain with the internal power-type cusp is an example of a uniform domain which fails to be NTA-domain.
Uniform domains are necessarily John domains, the latter one enclosing e.g. bounded domain satisfying the interior cone condition. See N\"akki--V\"ais\"al\"a~\cite{nv91} and V\"ais\"al\"a~\cite{vaisala88} for further information on uniform and John domains.

Recall, that a \emph{quasihyperbolic distance} $k_{\Om}$ between points $x,y$ in a domain $\Om \subsetneq \R^n$ is defined as follows
\begin{equation}\label{eq:quasi}
 k_{\Om}(x,y)=\inf_{\gamma}\int_{\gamma} \frac{ds(t)}{d(\gamma(t), \partial \Om)},
\end{equation}
where the infimum is taken over all rectifiable curves $\gamma$ joining $x$ and $y$ in $\Om$. Any two points in a uniform domain $\Om$ can always be join by at least one \emph{quasihyperbolic geodesic}, i.e. a curve for which the above infimum can be achieved. See Bonk-Heinonen-Koskela~\cite[Section 2]{bhk} and Gehring-Osgood~\cite{go} for more information.

We end this section by recalling the following geometric definition.
\begin{defn}\label{def:ball}
A domain $\Om\subset \R^n$ is said to satisfy the \emph{interior ball condition}
with radius $r_i > 0$ if
for every $w \in \partial \Omega$ there exists $\eta^i \in \Omega$
such that $B(\eta^i, r_i) \subset \Omega$ and $\partial B(\eta^i, r_i) \cap \partial \Omega = \{w\}$.
Similarly, a domain $\Om\subset \R^n$ is said to satisfy the \emph{exterior ball condition}
with radius $r_e > 0$ if
for every $w \in \partial \Omega$ there exists $\eta^e \in \R^n \setminus \Omega$
such that $B(\eta^e, r_e) \subset \R^n \setminus \Omega$ and $\partial B(\eta^e, r_e) \cap \partial \Omega = \{w\}$.
A domain $\Om\subset \R^n$ is said to satisfy the \emph{ball condition} with radius $r_b$
if it satisfies both the interior ball condition and the exterior ball conditions with radius $r_b$.
\end{defn}
It is well known that $\Om\subset \R^n$ satisfies the ball condition if and only if $\Om$ is a $C^{1,1}$-domain.
See Aikawa--Kilpel\"ainen--Shanmugalingam--Zhong \cite[Lemma 2.2]{AKSZ07} for a proof.
We also note that if $\Om\subset \R^n$ satisfies the ball condition then $\Om$ is a NTA-domain and hence also a uniform domain.

 Throughout the paper, unless otherwise stated, $c$ and $C$ will denote constants whose values may vary at each occurrence. If $c$ depends on the parameters $a_1, \dots, a_n$ we sometimes write $c(a_1, \dots, a_n)$.
When constants depend on the variable exponent $\px$ we write "depending on $p^-, p^+, c_{log}$"
in place of "depending on $p$" whenever dependence on $p$ easily reduces to $p^-, p^+, c_{log}$.


\section{Oscillation and Carleson estimates for $\px$-harmonic functions}\label{Sect-Carleson}

\setcounter{theorem}{0}
\setcounter{equation}{0}

 This section is devoted to discussing some important auxiliary results used throughout the rest of the paper. Namely, in Lemmas~\ref{osc-lem}, \ref{osc-Lukk} and~\ref{holderkont} we study oscillations of $\px$-harmonic functions over the balls intersecting the boundary of the underlying domain. We also employ geometric concepts such as NTA and uniform domains, quasihyperbolic geodesics and distance together with the Harnack inequality to obtain a supremum estimate for a $\px$-harmonic function over a chain of balls. Such estimates, discussed in $p=const$ setting for instance in Aikawa-Shanmugalingam~\cite{ASh} or Holopainen--Shanmugalingam--Tyson~\cite{HShT}, require extra attention for variable exponent $\px$ as now constant in the Harnack inequality depends on a $\px$-harmonic function and the inequality is non-homogeneous. In Theorem~\ref{Carleson} we show the main result of this section, namely the variable exponent Carleson estimate. Such estimates play crucial role in studies of positive $p$-harmonic functions, see e.g. Aikawa-Shanmugalingam~\cite{ASh}, also Garofalo~\cite{ga} for an application of Carleson estimates for a class of parabolic equations. According to our best knowledge Carleson estimates in the setting of equations with nonstandard growth have not been known so far in the literature. We apply Lemma~\ref{Carleson} in the studies of $\px$-harmonic measures in Section 5. Moreover, the geometry of the underlying domain turns out to be important in our investigations, in particular properties of NTA domains and uniform $\px$-fatness of the complement come into play.

We begin with recalling the Harnack estimate for $\px$-harmonic functions.

\begin{lemma}[Variable exponent Harnack inequality]\label{harnack}
Let $p$ be a bounded log-H\"older continuous variable exponent. Assume that $u$ is a nonnegative $\px$-harmonic
function in $B(w,4r)$, for some $w \in \R^n$ and $0 < r <\infty$.
Then there exists a constant $c_H$, depending on $n, p$ and $\sup_{\Om\cap
B(w,4r)} u$, such that
\begin{align*}
\sup_{B(w, r)} u \leq c_H\,\big( \inf_{B(w, r)} u + r \big).
\end{align*}
\end{lemma}


\begin{rem}\label{Remark_2_5}
The variable Harnack inequality in the above form was proved by
Alkhutov~\cite{alk1} (see also
Alkhutov--Krasheninnikova~\cite{alk2}) and subsequently improved to
embrace the case of unbounded solutions by
Harjulehto--Kinnunen--Lukkari~\cite[Theorem 3.9]{HKL07}. There,
$c_H$ depends only on $n, p$ and the $L^{q's}(B(w,4r))$-norm of
$u$ for $1<q<\frac{n}{n-1}$ and $s > p^+_{B(w,4r)} - p^-_{B(w,4r)}$.
\end{rem}


In what follows we will often iterate the Harnack inequality and
therefore we need to carefully estimate the growth of constants
involved in such iterations. Let $\Om\subset \R^n$ be a uniform domain with constant $M_{\Om}$ (for the definition of uniform domains and related concepts see the discussion in the end of Section~\ref{sect:prelim}).
We follow the argument in the proof of Lemma~3.9 in
Holopainen--Shanmugalingam--Tyson~\cite{HShT} and note that a
quasihyperbolic geodesic joining two points in $\Omega$ is
an $M'$-uniform curve with $M'$ depending only on $M_{\Om}$, cf. discussion
in Gehring--Osgood~\cite{go}. Let now $x$ and $y$ be given points in $B(w, \frac{r}{M'})\cap \Om$
for $w \in \bdry \Om$ and some fixed $r>0$. As in \cite{HShT} we
find a sequence of balls $B_i$, $i=1,\ldots, N$ covering
quasihyperbolic geodesic $\gamma$ joining $x$ and $y$ in $\Om$ (such
a geodesic always exists for points in uniform domains, see
discussion preceding the proof of \cite[Lemma~3.9]{HShT}) and
satisfying the following conditions (recall that $k_{\Om}(x,y)$
stands for a quasihyperbolic distance between points $x$ and $y$ and is given in \eqref{eq:quasi}):
\begin{itemize}
\item[(1)] $B_i\cap B_{i+1}\not = \emptyset$ for each $i$,
\item[(2)] $2B_i\subset B(w, 4r)\cap \Omega$,
\item[(3)] $N\leq 3k_{\Om}(x,y)$.
\end{itemize}

We estimate the quasihyperbolic distance $k_{\Om}(x,y)$ similarly as in formula (16) in
Aikawa--Shanmugalingam~\cite[Section 4]{ASh}.
Among other facts we employ the definition of John curve.
Assume that $d(x,\bdry \Om) \leq d(y,\bdry \Om)$ and note that then for a John curve $\gamma$, parametrized by arc-length so that $\gamma(0) = x$ and $\gamma(l(\gamma)) = y$, the following is true.
For all $z \in \gamma$ we have $ M_{\Om} d(z,\bdry \Om) \geq l(\gamma_{xz})$,
where $\gamma_{xz}$ is the sub curve from $x$ to $z$.
Using this we see that
\begin{align*}
 k_{\Om}(x,y)&\leq \int_{\gamma}\frac{ds(t)}{d(\gamma(t),\bdry \Om)}\leq \int_{0}^{\frac12 d(x,\bdry \Om)}\frac{ds}{\frac12 d(x,\bdry \Om)}+ \int_{\frac12 d(x,\bdry \Om)}^{l(\gamma)}\frac{ds}{ d(\gamma(t),\bdry \Om)} \\
&\leq  1+ M_{\Om}\int_{\frac12 d(x,\bdry \Om)}^{l(\gamma)}\frac{ds}{s}
=1+M_{\Om}\log s\,\rvert_{_{\frac12 d(x,\bdry \Om)}}^{^{l(\gamma)}}\leq
1+M_{\Om}\log s\,\rvert_{_{\frac12 d(x,\bdry \Om)}}^{^{M_{\Om}d(y,\bdry \Om)}} \\
 &= 1+M_\Om^2 +M_\Om\log 2 +M_\Om\log\left(\frac{d(y,\bdry \Om)}{d(x,\bdry \Om)}\right).
\end{align*}
Combining this with the estimate for the number of balls $N$ we get
\begin{equation}\label{N-est}
 N\leq 9M_{\Om}^2 +3M_\Om \log\left(\frac{d(y,\bdry \Om)}{d(x,\bdry \Om)}\right),
\end{equation}
whenever  $d(x,\bdry \Om) \leq d(y,\bdry \Om)$.
This estimate can be used in the iteration of Harnack inequality as follows.

Suppose that $x,y\in B(w, \frac{r}{M'})$. Then by the variable exponent Harnack inequality (Lemma \ref{harnack}) and the construction of the chain of balls $B_i$ above, we have that
\begin{align*}
 u(x)&\leq \sup u(x)_{B_1(x, r_1)}\leq c_H\,\big( \inf_{B_1} u + r_1 \big)\leq \ldots \leq \nonumber \\
&\leq c_H^N u(y)+c_H^N r_1+c_H^{N-1}r_2+\ldots+r_N\leq c_H^N u(y)+c_H^N Nr \nonumber \\
&\leq C^N \big( u(y)+r).
\end{align*}
By using (\ref{N-est}) we find that
\begin{align}\label{ineq:N-indep}
C^N\leq C^{9M_{\Om}^2 +3M_\Om \log\left(\frac{d(y,\bdry \Om)}{d(x,\bdry \Om)}\right)}\leq
C^{9M_{\Om}^2}C^{\log\left(\frac{d(y,\bdry \Om)}{d(x,\bdry \Om)}\right)^{3M_\Om}}\leq
C^{9M_{\Om}^2}\left(\frac{d(y,\bdry \Om)}{d(x,\bdry \Om)}\right)^{3M_\Om\log C},
\end{align}
whenever  $d(x,\bdry \Om) \leq d(y,\bdry \Om)$.


In some results of this section we appeal to notion of \emph{uniform $\px$-fatness}.
For the sake of completeness of the presentation we recall necessary definitions,
cf. Lukkari~\cite[Sections 3 and 4]{L10} and Holopainen--Shanmugalingam--Tyson~\cite[Section 3]{HShT} .

\begin{defn}
We say that $\Om$ has \emph{uniformly $\px$-fat complement},
if there exist a radius $r_0>0$ and a constant $c_0>0$ such that
\begin{equation}\label{ineq:fatness}
 \capp_{\px} ((\R^n\setminus \Om) \cap B(x, r), B(x, 2r)) \geq c_0\,\capp_{\px}(\overline{B}(x, r), B(x, 2r))
\end{equation}
for all $x \in \R^n\setminus \Om$ and all $r\leq r_0$.
\end{defn}

The next lemma provides an oscillation estimate. Similar result was proven by  Lukkari in \cite[Proposition 4.2]{L10}.
However, here we adapt the discussion from \cite{L10} to our case, for instance we do not require the boundary data to be H\"older continuous.

\begin{lemma}\label{osc-lem}
Let $\Omega \subset \Rn$  be a domain having a uniformly $\px$-fat complement with constants $c_0$ and $r_0$.
Let further $p$ be a bounded log-H\"older continuous
variable exponent satisfying either $p^+ \leq n$ or $p^- > n$.
Suppose that $w\in \partial \Om$, $r>0$ and
$u$ is a $\px$-harmonic function in $\Om \cap B(w, r)$,
continuous on $\overline{\Om} \cap \overline{B}(w, r)$ with $u = 0$ on $\partial\Om \cap B(w, r)$.
Then there exist $\beta$, $0<\beta\leq 1$, a constant $c>0$ and a radius $\hat r$ such that
\begin{equation*}
\sup_{B(w, \rho)\cap \Om} u \leq c \bigg(\frac{\rho}{r}\bigg)^{\beta}\big(\sup_{B(w, r)\cap \Om} u+r\big)
\end{equation*}
for all $\rho\leq r/2$ and $r \leq \hat r$.
The constants $\beta$ and $c$ depend on $n, p, \sup_{B(w, r)\cap \Om} u$ and $c_0$,
while $\hat r$ depends on $n, p^-, p^+, c_{log}$ and $r_0$.
\end{lemma}

\begin{proof}
Denote $p_0:= p(w)$ and split the discussion into two cases: $p_0>n$ and $p_0 \leq n$.
We start by proving the lemma for $p_0 > n$.
By assumptions $u$ is continuous on $\overline{B}(w, r) \cap \overline{\Om}$ with $u = 0$ on $B(w, r) \cap \partial\Omega$. Hence, we may use Theorem 1.2 in Alkhutov--Krasheninnikova~\cite{alk2},
with $D = B(w, r)\cap \Omega$ and $f=u$. In a consequence $f(w) = 0$ and ${\rm osc}_{\partial D} f \leq \sup_{B(w, r)\cap \Omega} u$ and we obtain that there exists $c = c(n, p, \sup_{B(w, r)\cap \Om} u)$
such that
\begin{align}
\sup_{B(w,\rho)\cap \Om} u \leq c \left( \frac{\rho}{r} \right)^{1-n/p_0}  \sup_{B(w, r)\cap \Om} u,
\end{align}
for all $\rho \leq r/4$ and with $r \leq \hat r(n, p^+, p^-,c_{log},r_0)$.
The dependence of $\hat r$ on the listed parameters follows from the proof of Theorem 1.2 in \cite{alk2}.
Hence, we conclude the lemma for $p_0>n$ by taking $\beta = \beta(p_0,n) = 1-n/p_0$.

Assume now that $p_0 \leq n$.
To prove the lemma in this case we will follow the steps and notation of the proof of Proposition 4.2 in Lukkari \cite{L10}. In the applications of Lemma~\ref{osc-lem} we will need to understand the exact dependence on constants and, therefore we repeat parts of the proof from \cite{L10}.

Let $\eta\in C_{0}^{\infty}(B(w, r))$. Then $\eta u_{+}\in W^{1,\px}_{0}(B(w, r) \cap \Om)$. Further, $\sup_{B(w, r)\cap \Om} u_{+}= \sup_{B(w, r)\cap \Om} u$ as $u$ attains the boundary values $u\equiv 0$ continuously on $B(w, r)\cap \partial \Om$. As in Lukkari's proof we define $\phi(r):=\sup_{B(w,r)\cap \partial \Om}u-u(w)$ and $\lambda(r):=\sup_{B(w,r)\cap \partial \Om}u$ and note that under our assumptions $\phi \equiv \lambda \equiv 0$.
Then we use \cite[Formula (3.4)]{L10} and \cite[Formula (4.2)]{L10} which requires $r \leq \hat r(n, p^-,p^+)$,
cf. Formula (3.2) in \cite{L10}.
Namely, \cite[Formula (3.4)]{L10} in our case reads
\begin{equation}\label{osc-lem-est1}
(\sup_{B(w, r)\cap \Om} u+r)C^{-1}\gamma(r)\leq \sup_{B(w, r)\cap \Om} u-\sup_{B(w, r/2)\cap \Om} u+r.
\end{equation}
The analysis of the proof of \cite[Formula (3.4)]{L10} and the proof of \cite[Theorem 3.3]{L10} reveals that
\begin{equation*}\label{osc-lem-est2}
(\sup_{B(w, r)\cap \Om} u+r)^{p(w)-p(x)}\leq c(c_{log})(\sup_{B(w, r)\cap \Om} u+1)^{p^+-p^-}:=C.
\end{equation*}
The $\px$-fatness of the complement of $\Om$ together with the capacity estimate \eqref{cap-est} imply the following inequalities (cf. \cite[Formula (3.5)]{L10} and \cite[Formula (4.2)]{L10}):
\begin{align*}
\gamma(r):=\left(\frac{\capp_{\px}((\R^n\setminus \Om)\cap B(w, r/2), B(w, r))}{r^{n-p(w)}}\right)^{\frac{1}{p(w)-1}}&\geq
\left(\frac{c_0\,\capp_{\px}(\overline{B}(w, r/2), B(w, r))}{r^{n-p(w)}}\right)^{\frac{1}{p(w)-1}} \nonumber \\
&\geq (c_0 c(n,p))^{\frac{1}{p(w)-1}}.
\end{align*}
Thus $\gamma_0:=C^{-1}\gamma(r)$ satisfies $c(c_0, n, p, \|u^+\|_{L^{\infty}(B(w, r)\cap \Om)})<\gamma_0<1$ and \eqref{osc-lem-est1} reads:
\begin{equation*}
\sup_{B(w, r/2)\cap \Om} u\leq \gamma_1(\sup_{B(w, r)\cap \Om} u+r),
\end{equation*}
where $\gamma_1:=\max\{\gamma_0,1-\gamma_0\}<1$. This inequality is a counterpart of \cite[Formula (4.3)]{L10}.
Note also that $\frac12\leq \gamma_1 <1$. We iterate the above inequality to obtain
\[
  \sup_{B(w, \frac{r}{2^m})\cap \Om} u \leq \gamma_1^m\big( \sup_{B(w, r)\cap \Om} u+c(\gamma_1)r\big),
\]
where $c(\gamma_1)<1$ if $\gamma_1=\frac12$ and $c(\gamma_1)\leq \frac{2\gamma_1^{m+1}}{2\gamma_1-1}$ for the remaining values of $\gamma_1\in (\frac12,1)$. We continue as in \cite{L10} to find that for $\beta=\log_{2}(\frac{1}{\gamma_1})$ it holds
\[
\gamma_1^m\leq 2^\beta \bigg(\frac{\rho}{r}\bigg)^{\beta},
\]
where $\beta$ depends on $c_0, n, p$ and $\sup_{B(w, r)\cap \Om} u$.
Hence, the proof is completed.
\end{proof}

\noindent To prove H\"older continuity up to the boundary we will also use the following oscillation estimate
which follows from Theorem 4.2, Lemma 2.8 in Fan-Zhao~\cite{Fan2}
and Lemma 4.8 in Ladyzhenskaya-Ural'tseva~\cite{lad-ur}.
The careful scrutiny of the presentation in \cite{Fan2} reveals the dependance of $c$ and
$\kappa$ on $\sup_{\Om} u$ and structure constants (cf. lemma below).
A similar result is given by Theorem 2.2 in Lukkari \cite{L10}, but under the assumption that $p^+ \leq n$.

\begin{lemma}
\label{osc-Lukk}
Let $p$ be a bounded log-H\"older continuous variable exponent and
let $u$ be a $\px$-harmonic function in $\Om$ and let $B(w,r)\Subset \Om$.
Then there exist $c$ and $\kappa$, $0 < \kappa < 1$, such that for all $0 < \rho \leq r$ it holds that
\begin{equation*}
{\rm osc}_{B(w,\rho)} u\leq c \bigg(\frac{\rho}{r}\bigg)^{\kappa}\big({\rm osc}_{B(w,r)} u + r\big).
\end{equation*}
The constants $c$ and $\kappa$ depend on $n, p^+, p^{-}$ and $\sup_{\Om} u$.
\end{lemma}

We are now ready to formulate the version of H\"older continuity up to the boundary which will be needed in this paper.

\begin{lemma}\label{holderkont}

Let $\Omega \subset \Rn$  be a domain having a uniformly $\px$-fat complement with constants $c_0$ and $r_0$.
Let further $p$ be a bounded log-H\"older continuous
variable exponent.
Suppose that $w\in \partial \Om$, $r>0$ and
$u$ is a $\px$-harmonic function in $\Om \cap B(w, 2r)$,
continuous on $\overline{\Om} \cap \overline{B}(w, 2r)$ with $u = 0$ on $\partial\Om \cap B(w, 2r)$.
Let $\gamma=\min\{\kappa, \beta\}$ and $r < \hat r$ for $\beta$ and $\hat r$ as in Lemma~\ref{osc-lem}
and $\kappa$ as in Lemma~\ref{osc-Lukk}.
Then there exists $C>0$
such that
\begin{equation*}
|u(x)-u(y)|\leq C\bigg(\frac{|x-y|}{r}\bigg)^{\gamma}\bigg(\sup_{B(w, 2r)\cap \Om} u+r\bigg)
\quad \textrm{whenever} \quad x,y \in B(w, r)\cap \Omega.
\end{equation*}

\end{lemma}

\begin{proof}
Let $x,y \in B(w, r)\cap \Om$ and let $x_0\in \partial \Om$ be such that $d(x, \partial \Om)=|x-x_0|$.
We distinguish two cases.

\emph{Case 1.} $|x-y|<\frac12 d(x,\partial \Om)$. Then Lemma~\ref{osc-Lukk} applied with $\rho=|x-y|$ and
$r=\tau/2$ for $\tau= d(x,\partial \Om)$ together with Lemma~\ref{osc-lem} imply the following inequalities:
\begin{align*}
|u(x)-u(y)|&\leq {\rm osc}_{B(x, \rho)}u \leq c \bigg(\frac{|x-y|}{\tau/2}\bigg)^{\kappa}\big({\rm osc}_{B(x, \tau/2)} u+\frac{\tau}{2}\big) \nonumber \\
&\leq c 2^{\kappa}\bigg(\frac{|x-y|}{\tau}\bigg)^{\kappa}\big({\rm osc}_{B(x_0,\frac{3}{2}\tau)\cap \Om} u+\frac{\tau}{2}\big) \nonumber \\
&\leq c 2^{\kappa} \bigg(\frac{|x-y|}{\tau}\bigg)^{\kappa}\big(\sup_{B(x_0,\frac{3}{2}\tau)\cap \Om} u+\frac{\tau}{2}\big) \\
&\leq c 2^{\kappa} \bigg(\frac{|x-y|}{\tau}\bigg)^{\kappa}\bigg[2^\beta\left(\frac{\frac{3}{2}\tau}{2r}\right)^\beta\big(\sup_{B(w, 2r)\cap \Om} u+2r\big)+\frac{\tau}{2}\bigg] \nonumber \\
&\leq c 3^\beta 2^{\kappa-\beta} \bigg(\frac{|x-y|}{\tau}\bigg)^{\kappa} \bigg(\frac{\tau}{r}\bigg)^{\beta}\bigg(\sup_{B(w, 2r)\cap \Om} u+2r+r^\beta \tau^{1-\beta}\frac12\bigg) \nonumber \\
&\leq C \bigg(\frac{|x-y|}{r}\bigg)^{\kappa} \bigg(\frac{\tau}{r}\bigg)^{\beta-\kappa}\bigg(\sup_{B(w, 2r)\cap \Om} u+2r\bigg). \nonumber
\end{align*}
If $\beta-\kappa>0$, then $\left(\frac{\tau}{r}\right)^{\beta-\kappa}<1$ and we get the assertion for $\gamma=\kappa$. Otherwise, if $\beta-\kappa\leq 0$, then since $|x-y|<\frac12\tau$, we have that
\[
\bigg(\frac{|x-y|}{r}\bigg)^{\kappa} \bigg(\frac{\tau}{r}\bigg)^{\beta-\kappa}< \bigg(\frac{|x-y|}{r}\bigg)^{\kappa}
\bigg(\frac{r}{2|x-y|}\bigg)^{\kappa-\beta}\leq 2^{\beta-\kappa} \bigg(\frac{|x-y|}{r}\bigg)^{\beta}.
\]
Thus, the estimate holds for $\gamma=\min\{\kappa, \beta\}$.

\emph{Case 2.} $|x-y|\geq \frac12 d(x,\partial \Om)$.   Since $u(x_0)=0$, we have by Lemma~\ref{osc-lem} that
\begin{align*}
|u(x)-u(y)|&\leq |u(x)-u(x_0)|+|u(y)-u(x_0)|\nonumber \\
&\leq 2^{\beta} \bigg(\frac{|x-x_0|}{r}\bigg)^{\beta}\big(\sup_{B(x_0, r)\cap \Om} u+r\big)
+ 2^{\beta} \bigg(\frac{|y-x_0|}{r}\bigg)^{\beta}\big(\sup_{B(x_0, r)\cap \Om} u+r\big) \\
&\leq 4^{\beta} \bigg(\frac{|x-y|}{r}\bigg)^{\beta}\bigg(\sup_{B(x_0, r)\cap \Om} u+r\bigg) \nonumber \\
&\leq 4^{\beta} \bigg(\frac{|x-y|}{r}\bigg)^{\beta}\bigg(\sup_{B(w, 2r)\cap \Om} u+r\bigg).
\end{align*}
Since $|x-y|<r$, the last inequality holds as well with exponent $\gamma=\min\{\kappa, \beta\}$,
giving us the assertion of the lemma in this case. The proof of Lemma~\ref{holderkont} is, therefore, completed.
\end{proof}

Following the proof of Theorem~6.31 in~\cite{hkm} one can show that
\emph{if the complement of $\Om$ satisfies the corkscrew condition at $w\in \partial \Om$,
then $\R^n\setminus \Om$ is $\px$-fat at $w$}.
Indeed, using the elementary properties of the relative $\px$-capacity (see Section~10.2 in Dieninig--Harjulehto--H\"ast\"o--R\r u\v zi\v cka~\cite{DHHR}, in particular Lemma 10.2.9 in~\cite{DHHR} and the discussion following it) one shows that \eqref{ineq:fatness} holds at $w$.
Here the log-H\"older continuity of $\px$ plays an important role as one also employs property \eqref{logH-ball2}.
Hence, the complement of a NTA domain is uniformly $\px$-fat, see Definition \ref{def:NTA}.

We are now in a position to prove the main result of this section,
the Carleson-type estimate.

\begin{theorem}[Variable exponent Carleson-type estimate]\label{Carleson}
Assume that $\Omega\subset\R^{n}$ is an NTA domain with constants
$M_{\Omega}$ and $r_{\Omega}$.
Let $w \in\partial\Omega$, $0<r\leq r_\Om$ and $p$ be a bounded log-H\"older continuous variable exponent
satisfying either $p^+ \leq n$ or $p^- > n$. Suppose that $u$ is a positive $\px$-harmonic function in $\Omega \cap
B(w,r)$, continuous on $\bar \Omega \cap B(w,r)$ with $u = 0$ on
$\partial \Omega \cap B(w,r)$. Then there exist constants $c$ and
$c'$ such that
\begin{align*}
\sup_{\Omega \cap B(w, r')} u \leq c \left( u(a_{r'}(w)) + r' \right),
\end{align*}
where $r' = r / c'$.
The constant $c$ depends on $n, p, \sup_{B(w, r)\cap \Om} u$ and $M_{\Omega}$
while $c'$ depends on $n, p^-, p^+, c_{log}$ and $M_{\Omega}, r_\Om$.
\end{theorem}

\begin{proof}
We proceed following the main lines of Caffarelli--Fabes--Mortola--Salsa~\cite{CFMS}.
Let $k$ be a large number to be determined later and assume that
\begin{align}\label{eq:assumption-k}
k \left( u(a_{r'}(w)) + r' \right) < \sup_{\Omega \cap B(w,  r')} u = u(x_1)
\end{align}
where $x_1 \in \partial B(w, r') \cap \Omega$ by the maximum principle.
We want to derive a contradiction if $k$ is chosen large enough.

Suppose first that $d(x_1,\partial \Omega) \geq  r' / 100$. Since $\Om$ is an NTA domain, it is in particular uniform. Hence, we may assume that $r'$ is so small that any two points in $B(w, 2 r') \cap \Omega$ can be connected by a Harnack chain
totally contained in $B(w, r)\cap \Omega$. Then $r' = r/c'$ depends
only on $M_{\Omega}$ and $r_\Om$.
Since the $L^{\infty}$-norm of $u$ is bounded in $B(w,r)\cap \Omega$,
we can iterate Harnack's inequality using the same constant
for each ball contained in $B(w,r)\cap \Omega$. Thus, the Harnack
inequality yields the existence of a constant $c_0$, which by
\eqref{ineq:N-indep} depends only on $c_H$ and $M_{\Omega}$, and
such that
\begin{align}\label{eq:step1case1}
u(x_1) \leq c_0 \left( u(a_{r'}(w)) + r' \right).
\end{align}
This gives us a contradiction if $k > c_0$ and hence the proof of
Theorem \ref{Carleson} follows in the case when $d(x_1,\partial
\Omega) \geq  r' / 100$.

Next, assume that $d(x_1,\partial \Omega) <  r' / 100$.
It follows by the Harnack inequality and discussion before \eqref{ineq:N-indep} that there exist constants $\hat c,\, \lambda \in [1,\infty)$, depending only on $M_{\Omega}$ and $c_H$, such that
\begin{align}\label{eq:harnack-to-boundary}
u(x_1) \leq \hat c \left( \frac {d(a_{r'}(w), \partial \Omega)}{d(x_1, \partial \Omega)} \right)^{\lambda}
\left( u(a_{r'}(w)) + r' \right).
\end{align}
From (\ref{eq:assumption-k}) and (\ref{eq:harnack-to-boundary}) we see that
\begin{align}\label{eq:dist-bound}
\frac{d(x_1, \partial \Omega)}{d(a_{r'}(w), \partial \Omega) } <  \left(\frac {\hat c}{k} \right)^{1/{\lambda}}.
\end{align}
Let $x_1^+ \in B(w, r') \cap \partial\Omega$ be a point minimizing $|x_1^+ - x_1|$.
By decreasing $r'$ if necessary, we apply Lemma \ref{holderkont} for $B(x_1^+, r'/2)$ to obtain

\begin{align}\label{eq:holder-gives}
u(x_1)-u(x_1^+)= u(x_1)\leq C\, \bigg(\frac{d(x_1, \partial \Omega)}{{r'}/4}\bigg)^{\gamma} \bigg(\sup_{B(x_1^+, {r'}/2 )\cap \Om} u+\frac{r'}4\bigg),
\end{align}
where $\gamma$ and $C$ depend on $n, p, \sup_{B(w, r)\cap \Om} u$ and $M_\Om$.
The constant $c'$ now depends on $n, p^-, p^+, c_{log}, M_\Om$ and $r_\Om$.
By using the Harnack inequality, the maximum principle, assumption
\eqref{eq:assumption-k} together with (\ref{eq:dist-bound}) and
(\ref{eq:holder-gives}) we obtain, for some $x_2 \in \partial
B(x_1^+,  r'/2)\cap\Omega$, the existence of $\check c = \check
c(c_H, M_{\Omega})$ such that
\begin{align}\label{eq:inequalityloong}
k \check c^{-1} \left( u(a_{r' / 2}(x_1^+)) + r' \right)
&\leq k \check c^{-1} \left(\check c \left[ u(a_{r'}(w)) + r' \right] +r'\right) = k u(a_{r'}(w)) + k r' + \frac{k}{\check c} r' \nonumber\\
&< k\left(1+\frac{1}{\check c}\right)\left(u(a_{r'}(w))+r'\right) < \left(1+\frac{1}{\check c}\right) u(x_1) \nonumber\\
&\leq \left(1+\frac{1}{\check c}\right)C  \bigg(4\frac{d(x_1, \partial \Omega)}{{r'}}\bigg)^{\gamma} \bigg(u(x_2)+\frac{r'}4\bigg) \nonumber\\
&\leq \left(1+\frac{1}{\check c}\right)C \left(\frac{\hat c}{k}\right)^{\gamma/\lambda} \left(u(x_2)+\frac{r'}4\right).
\end{align}
In the last inequality, we have also used $d(a_{r'}(w),\partial \Omega) \leq r'$.
Define constant $k_1$ such that
\begin{align*}
\check c  \left(1+\frac{1}{\check c}\right)C \left(\frac{\hat c}{k_1}\right)^{\gamma/\lambda} = 1.
\end{align*}
By demanding $k > \max\{c_0, k_1\}$ we obtain
\begin{align*}
k \left(u(a_{\frac{r'}2}(x_1^+)) + r'\right) \leq u(x_2) + \frac{r'}4 \qquad \textrm{and} \qquad u(x_1) < u(x_{2})+\frac{r'}4.
\end{align*}
Let $k > 1$. Then $kr'/2 \geq r'/4$ and the above inequalities take the following form:
\begin{equation}\label{eq:k-repeat}
k \left(u(a_{\frac{r'}2}(x_1^+)) + \frac{r'} 2 \right) \leq u(x_2)  \qquad \textrm{and} \qquad u(x_1) < u(x_{2})+\frac{r'}4.
\end{equation}
We will now repeat the above argument starting from
\eqref{eq:assumption-k} with \eqref{eq:k-repeat} replacing
\eqref{eq:assumption-k}. As now the initial condition has an
additional term on the right-hand-side, we provide details of the
reasoning. Once those are explained, it will become more apparent
how to continue with the recurrence argument. Suppose first that
$d(x_2,\partial \Omega) \geq  r'/200$. Then, similarly as for
$x_1$ we get from \eqref{eq:k-repeat} and the Harnack inequality
that
\begin{equation*}
k \left(u(a_{\frac{r'}2}(x_1^+)) + \frac{r'} 2 \right) \leq u(x_2)  \leq c_0 \left(u(a_{\frac{r'}2}(x_1^+)) + \frac{r'} 2\right),
\end{equation*}
where $c_0$ is the constant from \eqref{eq:step1case1}. Hence, we
again obtain the contradiction if $k > c_0$.

Let now  $d(x_2,\partial \Omega) < r' /200$.
The discussion similar to that for \eqref{eq:harnack-to-boundary} gives us
\begin{equation}\label{eq:harnack-to-boundary2}
u(x_2) \leq \hat c \left( \frac {d(a_{\frac{r'}2}(x_1^+), \partial \Omega)}{d(x_2, \partial \Omega)} \right)^{\lambda} \left( u(a_{\frac{r'}2}(x_1^+)) + \frac {r'} 2 \right).
\end{equation}
From \eqref{eq:k-repeat} and \eqref{eq:harnack-to-boundary2} we see that
\begin{align*}
\frac{d(x_2, \partial \Omega)}{d(a_{\frac{r'}2}(x_1^+), \partial \Omega) } <  \left(\frac {\hat c}{k} \right)^{1/{\lambda}}.
\end{align*}
We take point $x_2^+\in B(x_1^+, \frac{r'}2)\cap \bdry \Om$ minimizing $|x_2-x_2^+|$ and then apply
Lemma \ref{holderkont} for $B(x_2^+, \frac{r'}4)$. In a result we get
\begin{align*}
u(x_2)\leq C\, \bigg(\frac{d(x_2, \partial \Omega)}{r'/8}\bigg)^{\gamma} \bigg(\sup_{B(x_2^+,{r'}/4)\cap \Om} u+\frac{r'}{8}\bigg).
\end{align*}
Following the same reasoning as in \eqref{eq:inequalityloong} we obtain,
for some $x_3 \in \partial B(x_2^+, \frac{r'}4)\cap\Omega$, that
\begin{align*}
k \check c^{-1} \left( u(a_{\frac{r'}4}(x_2^+))+ \frac{r'} 2 \right)
&\leq k \check c^{-1} \left(\check c \left[u(a_{\frac{r'}2}(x_1^+)) + \frac{r'}2\right] + \frac{r'} 2 \right) = k u(a_{\frac{r'}2}(x_1^+)) + \frac{k r'}{2} + \frac{k}{ 2 \check c} r' \\
&< k\left(1+\frac{1}{\check c}\right)\left(u(a_{\frac{r'}2}(x_1^+))+ \frac{r'} 2\right) < \left(1+\frac{1}{\check c}\right) u(x_2)\\
&\leq \left(1+\frac{1}{\check c}\right)C  \bigg(8\frac{d(x_1, \partial \Omega)}{{r'}}\bigg)^{\gamma} \bigg(u(x_3)+\frac {r'}8\bigg)\\
&\leq \left(1+\frac{1}{\check c}\right)C \left(\frac{\hat c}{k}\right)^{\gamma/\lambda} \left(u(x_3)+\frac {r'}8\right).
\end{align*}
Since $k > k_1$ and $k r' / 4 \geq  r' / 8$ we arrive at
\begin{equation*}
k \left(u(a_{\frac{r'}4}(x_2^+))+\frac{r'} 4\right) \leq u(x_3)  \qquad \textrm{and} \qquad u(x_2) < u(x_3)+\frac{r'}{8}.
\end{equation*}
Having established first two steps of the iteration, we now choose points $x_m$, $x_m^+$ in the similar way as we found $x_1, x_1^+$ and $x_2, x_2^+$ and get that
\begin{equation*}
k\left(u(a_{ r' / 2^{m}}(x_m^+)) + \frac{r'}{2^m} \right) \leq u(x_{m+1})  \qquad \textrm{and}
\qquad u(x_m) < u(x_{m+1})+\frac{r'}{2^{m+1}}.
\end{equation*}
If $m\to \infty$, then $x_m\to y\in \bdry \Om\cap B(w, 2 r')$.
Since $u$ is assumed continuous on $\overline{\Omega} \cap B(w,r)$ with $u = 0$ on $\partial\Omega \cap B(w,r)$
we obtain that $u(x_m)\to u(y) = 0$.
Hence we conclude that
\begin{align*}
k \left( u(a_{r'}(w)) + r' \right) &< u(x_1)<u(x_2)+\frac{r'}4< u(x_3)+\frac{r'}{8}+\frac{r'}4 <\ldots <\\
&< u(x_{m})+\frac{r'}2 \to \frac{r'}2 \quad \hbox{ for } m\to \infty.
\end{align*}
This gives
\[
k \left( u(a_{r'}(w)) + r' \right) < \frac{r'}2
\]
which leads to $k < 1/2$ and results in the contradiction by demanding $k > \max\{1,c_0,k_1\}$.
Thus the proof of Theorem \ref{Carleson} is completed.
\end{proof}

\section{Constructions of $\px$-barriers}\label{Sect-barriers}

\setcounter{theorem}{0}
\setcounter{equation}{0}

Below we present two types of barrier functions. The first
type is based on a work of Wolanski~\cite{wol}, however our
Lemma~\ref{le:comparisonfunctions_W} improves result of \cite{wol},
see Remark~\ref{Wolanski-comment}. We employ Wolanski-type barriers
in the upper and lower boundary Harnack estimates, see
Section~\ref{Sect-estimates}. The second type of barriers has been
inspired by a work of Bauman~\cite{bau} who uses barriers in studies
of a boundary Harnack inequality for uniformly elliptic equations
with bounded coefficients. Both approaches have advantages. On one
hand a radius of a ball on which a Wolanski-type barrier exists
depends on less number of parameters then a radius of a
corresponding ball for a Bauman-type barrier, but on the other hand
exponents in Wolanski-type barriers depend on larger number of
parameters than exponents in Bauman-type barriers, cf.
Lemmas~\ref{le:comparisonfunctions_W} and
\ref{le:comparisonfunctions}. Therefore, both types of barriers are
useful in applications.

\subsection{Upper and lower $\px$-barriers of Wolanski-type}
\begin{lemma}\label{le:comparisonfunctions_W}
Let $y\in \R^n$ and $r>0$ be fixed and let $p$ be a bounded
Lipschitz continuous variable exponent on $\overline{B}(y,2r)$. Let
further $M > 0$ be given and for $x\in B(y, 2r)$ define functions
\begin{align*}
\hat u(x) = \frac{M}{e^{-\mu} - e^{-4\mu}} \left( e^{-\mu} - e^{-\mu \frac{|x - y|^2}{r^2}} \right)
\quad \text{and} \quad
\check u(x) = \frac{M}{e^{-\mu} - e^{-4\mu}} \left(e^{-\mu \frac{|x - y|^2}{r^2}} - e^{-4\mu} \right).
\end{align*}
Then there exist $r_* = r_*(p^-, ||\nabla p||_{L^\infty})$ and
$\mu_* = \mu_*(p^+, p^-, n, ||\nabla p||_{L^\infty}, M)$ such that
$\hat u(x)$ is a $\px$-supersolution and $\check u(x)$ is a
$\px$-subsolution in $B(y, 2r) \setminus B(y, r)$ whenever $\mu \geq
\mu_*$ and $r \leq r_*$. Furthermore, it holds that
\begin{align*}
 &\hat u(x) = M \hbox{ on } \partial B(y, 2r)\quad \hbox{ and }\quad \hat u(x) = 0 \hbox{ on } \partial B(y, r)\\
 &\check u(x) = 0 \hbox{ on } \partial B(y, 2r)\quad \hbox{ and }\quad \check u(x) = M \hbox{ on } \partial B(y, r).
\end{align*}
\end{lemma}

\begin{rem}\label{Wolanski-comment}
 We would like to point out that the above theorem improves substantially some results on barrier functions in variable exponent setting, see Corollary~4.1 in Wolanski~\cite{wol}. Namely, in \cite{wol} the radius $r$ depends also on $M$ whereas here we manage to avoid such a dependence (see (\ref{eq:logbound}) and (\ref{eq:sista_super}) for details).
 This plays a role in the proof of Lemma~\ref{le:lower}.
\end{rem}

\begin{proof}
We begin the proof by noting that for any twice differentiable function $u$ we have
\begin{align*}
\Delta_{p(x)} u = \div \left( |\nabla u|^{p(x)-2} \nabla u \right) = \left\langle \nabla \left(|\nabla u|^{p(x)-2}\right), \nabla u \right\rangle + |\nabla u|^{p(x)-2} \Delta u.
\end{align*}
Now
\begin{align*}
&\left\langle \nabla \left(|\nabla u|^{p(x)-2}\right), \nabla u \right \rangle
= \sum_{i=1}^n \frac{\partial}{\partial x_i} \left(|\nabla u|^{p(x)-2}\right) \frac{\partial u}{\partial x_i}
= \sum_{i=1}^n \frac{\partial}{\partial x_i} \left(e^{(p(x)-2)\log\left( |\nabla u| \right)}\right) \frac{\partial u}{\partial x_i}\\
&= \sum_{i=1}^n \frac{\partial}{\partial x_i} \bigg\{(p(x)-2)\log\left( |\nabla u| \right)\bigg\} |\nabla u|^{p(x)-2} \frac{\partial u}{\partial x_i}\\
&= \sum_{i=1}^n \left\{ \frac{\partial p}{\partial x_i} \log\left( |\nabla u| \right) + (p(x)-2)\frac{1}{|\nabla u|} \frac{\partial}{\partial x_i} \left( |\nabla u| \right) \right\} |\nabla u|^{p(x)-2} \frac{\partial u}{\partial x_i}\\
&= \sum_{i=1}^n \left\{ \frac{\partial p}{\partial x_i} \log\left( |\nabla u| \right) + (p(x)-2)\frac{1}{|\nabla u|^2} \sum_{j=1}^n \frac{\partial u}{\partial x_j}\frac{\partial^2 u}{\partial x_j \partial x_i} \right\} |\nabla u|^{p(x)-2} \frac{\partial u}{\partial x_i}\\
&=  \left\{ \langle \nabla p, \nabla u \rangle \log\left( |\nabla u| \right) + (p(x)-2)\frac{1}{|\nabla u|^2} \Delta_{\infty} u \right\} |\nabla u|^{p(x)-2}.
\end{align*}
Moreover, assuming $|\nabla u| > 0$ we obtain the following:
\begin{align}
\Delta_{\px} u \leq(\geq) \,0 \iff \langle \nabla p, \nabla u \rangle \log\left( |\nabla u| \right) + (p(x)-2) \frac{\Delta_{\infty} u}{|\nabla u|^2}   +  \Delta u \leq (\geq)\, 0.\label{eq:normalizedplap1}
\end{align}
From \eqref{eq:normalizedplap1} we see that comparing to the
constant $p$ case we have the extra term involving no second
derivatives but the gradient of both $u$ and $\px$ instead.

We begin by showing that $\hat u$ is a supersolution. We will find $\mu$, $A$, $B$ and $r$ such that the function
\begin{align}
\hat u(x) = - A e^{-\mu \left( \frac{|x - y|}{r}\right)^2 } + B, \qquad \textrm{where} \quad r < |x - y| < 2 r
\label{eq:defnsuper}
\end{align}
has the desired properties. Differentiation of $\hat u$ yields
\begin{align}
\hat u_{x_i} = \frac{2 A \mu}{r^2} e^{-\mu \left( \frac{|x - y|}{r}\right)^2 } (x_i - y_i), \qquad | \nabla \hat u| =  \frac{2 A \mu}{r^2} e^{-\mu \left( \frac{|x - y|}{r}\right)^2 } |x - y|, \label{eq:supergrad}
\end{align}

\begin{align*}
\hat u_{x_i x_j} = \frac{2 A \mu}{r^2} e^{-\mu \left( \frac{|x - y|}{r}\right)^2 }
\bigg\{ \delta_{ij} - \frac{2 \mu}{r^2} (x_i - y_i)(x_j - y_j)\bigg\}.
\end{align*}
Next we observe that
\begin{align}
&\Delta \hat u =  \sum_{i=1}^n \hat u_{x_i x_i}
= \frac{2 A \mu}{r^2} e^{-\mu \left( \frac{|x - y|}{r}\right)^2 } \sum_{i=1}^n \bigg\{ 1 - \frac{2 \mu}{r^2} (x_i - y_i)^2\bigg\}
= \frac{2 A \mu}{r^2} e^{-\mu \left( \frac{|x - y|}{r}\right)^2 } \bigg\{ n - \frac{2 \mu}{r^2} |x - y|^2 \bigg\},
\label{eq:lapl}
\end{align}
and since $\sum_{i,j = 1}^n \delta_{ij} (x_i - y_i)  (x_j - y_j) = |x-y|^2$ and $\sum_{i,j = 1}^n  (x_i - y_i)^2  (x_j - y_j)^2 = |x-y|^4$
we also have
\begin{align}
&\Delta_{\infty} \hat u =  \sum_{i,j = 1}^n \hat u_{x_i x_j} \hat u_{x_i} \hat u_{x_j} \nonumber \\
&= \left(\frac{2 A \mu}{r^2}\right)^3 e^{-3\mu \left( \frac{|x - y|}{r}\right)^2 }
\sum_{i,j = 1}^n \bigg\{ (x_i - y_i)(x_j - y_j) \delta_{ij} - \frac{2 \mu}{r^2}(x_i - y_i)^2(x_j - y_j)^2\bigg\}
\label{eq:inftylap}\\
&= \left(\frac{2 A \mu}{r^2}\right)^3 e^{-3\mu \left( \frac{|x - y|}{r}\right)^2 }|x-y|^2
\bigg\{ 1 - \frac{2 \mu}{r^2}|x-y|^2\bigg\}. \nonumber
\end{align}
We collect expressions \eqref{eq:lapl} and \eqref{eq:inftylap} and insert them into \eqref{eq:normalizedplap1} to obtain the following inequality.
\begin{align*}
\langle \nabla p, \nabla u \rangle \log\left( |\nabla \hat u| \right) +
\frac{2 A \mu}{r^2}e^{-\mu \left( \frac{|x - y|}{r}\right)^2 }
\bigg\{(p(x)-2) \left(1 - \frac{2 \mu}{r^2}|x-y|^2\right) + n - \frac{2 \mu}{r^2}|x-y|^2 \bigg\} \leq 0.
\end{align*}
We simplify the above condition by using $\langle \nabla p, \nabla u \rangle =  \frac{2 A \mu}{r^2} e^{-\mu \left( \frac{|x - y|}{r}\right)^2 } \langle \nabla p,  x - y \rangle$:
\begin{align*}
\langle \nabla p,  x - y \rangle \log\left( |\nabla \hat u| \right)
- \frac{2 \mu}{r^2}|x-y|^2 (\p - 1) + n + \p - 2 \leq 0.
\end{align*}
This holds true if
\begin{align}\label{eq:normalizedplap_strong1}
2 r\|\nabla p\|_{L^\infty}  |\log\left( |\nabla \hat u| \right)| - { 2 \mu } (p^{-} - 1) + n + p^+ - 2 \leq 0.
\end{align}
Next we demand that our function $\hat u$ satisfies $\hat u(x) = M$ whenever $x \in \partial B(y, 2r)$
and $\hat u(x) = 0$ whenever $x \in \partial B(y, r)$.
These assumptions imply that $A = M/(e^{-\mu} - e^{-4\mu})$ and $B = M e^{-\mu}/(e^{-\mu} - e^{-4\mu})$.

We now bound $\log(|\nabla\hat u|)$. By \eqref{eq:supergrad} we have
\begin{align*}
| \nabla \hat u| =  \frac{2 M \mu}{e^{-\mu} - e^{-4\mu}} \frac{|x-y|}{r^2}  e^{-\mu \left( \frac{|x - y|}{r}\right)^2 },
\end{align*}
and hence, upon using $r < |x-y| < 2 r$, we obtain the following estimate
\begin{align*}
\frac{2 M \mu e^{-3\mu}}{r(1-e^{-3\mu})} \leq | \nabla \hat u| \leq  \frac{4 M \mu}{r(1-e^{-3\mu})}.
\end{align*}
Thus
\begin{align*}
- 3\mu + \log\left(\frac{2 M \mu}{r(1-e^{-3\mu})}\right) \leq \log \left(| \nabla \hat u|\right) \leq  \log\left(\frac{4 M \mu}{r(1-e^{-3\mu})}\right).
\end{align*}
Therefore, we conclude
\begin{align}\label{eq:logbound}
| \log \left(| \nabla \hat u|\right) | \leq  \log\left(\frac{4}{1-e^{-3\mu}}\right) + |\log M| + |\log r| + \log \mu + 3\mu.
\end{align}
Assume that $\mu$ is large and combine
\eqref{eq:normalizedplap_strong1} together with \eqref{eq:logbound}
to obtain that $\hat u$ is a supersolution provided that the
following condition is satisfied.
\begin{align}\label{eq:sista_super}
2 r \|\nabla p\|_{L^\infty} \left(\log\left(\frac{4}{1-e^{-3\mu}}\right) + |\log M| + |\log r|  +  4\mu \right) -  2 \mu (p^{-} - 1) + n + p^+ - 2 \leq 0.
\end{align}
Upon rearranging terms in \eqref{eq:sista_super} we obtain the following inequality:
\begin{align*}
\mu \left(8r \|\nabla p\|_{L^\infty} -2 (p^--1)\right) +2r\|\nabla p\|_{L^\infty}\left(\log\left(\frac{4}{1-e^{-3\mu}}\right) + |\log M| + |\log r|\right) + n + p^+ - 2 \leq 0.
\end{align*}
Pick $r_0 = (p^{-} - 1) / (4 \|\nabla p\|_{L^\infty} )$.  Then for
$r \leq r_0$ the above inequality can be satisfied by a large enough
$\mu$ (upon including the term $|\log M|$ into the first log-term).
Moreover, taking $r_* = \min\{r_0, 1/4\}$ ensures that $r |\log{r}|$
is an increasing function of $r$ for $r \leq r_*$. Thus we conclude
that if  $r \leq r_*$, then there exists $\mu_* = \mu_*(p^+, p^-, n,
||\nabla p||_{L^\infty}, M)$ such that $\hat u$ is a supersolution
for $\mu \geq \mu_*$. This completes the proof of the supersolution.

Next we want to show that $\check u$ is a subsolution.
We will find $\mu$, $C$, $D$ and $r$ in the function
\begin{align*}
\check u(x) =  C e^{-\mu \left( \frac{|x - y|}{r}\right)^2 } + D, \qquad \textrm{where} \quad r < |x - y| < 2 r.
\end{align*}
In this case
\begin{align}
\check u_{x_i} = - \frac{2 C \mu}{r^2} e^{-\mu \left( \frac{|x - y|}{r}\right)^2 } (x_i - y_i), \qquad | \nabla \check u| =  \frac{2 C \mu}{r^2} e^{-\mu \left( \frac{|x - y|}{r}\right)^2 } |x - y|, \label{eq:subsol}
\end{align}

\begin{align*}
\check u_{x_i x_j} = -\frac{2 C \mu}{r^2} e^{-\mu \left( \frac{|x - y|}{r}\right)^2 }
\bigg\{ \delta_{ij} - \frac{2 \mu}{r^2} (x_i - y_i)(x_j - y_j)\bigg\}.
\end{align*}
Similarly to computations in \eqref{eq:lapl} and \eqref{eq:inftylap} we observe that
\begin{align*}
\Delta \check u =
-\frac{2 C \mu}{r^2} e^{-\mu \left( \frac{|x - y|}{r}\right)^2 } \bigg\{ n - \frac{2 \mu}{r^2} |x - y|^2 \bigg\},
\end{align*}

\begin{align*}
\Delta_{\infty} \check u
= -\left(\frac{2 C \mu}{r^2}\right)^3 e^{-3\mu \left( \frac{|x - y|}{r}\right)^2 }|x-y|^2
\bigg\{ 1 - \frac{2 \mu}{r^2}|x-y|^2\bigg\},
\end{align*}
and
\begin{align*}
\langle \nabla p, \nabla u \rangle = - \frac{2 C \mu}{r^2} e^{-\mu \left( \frac{|x - y|}{r}\right)^2 } \langle \nabla p,  x - y \rangle.
\end{align*}
Collecting terms we obtain from \eqref{eq:normalizedplap1} that the condition for $\check u$ to be a subsolution becomes
\begin{align}\label{eq:normalizedplap_strong1ny}
-2 r\|\nabla p\|_{L^\infty}  |\log\left( |\nabla \check u| \right)| + 2 \mu (p^{-} - 1) - n - p^+ + 2 \geq 0.
\end{align}
This is equivalent to \eqref{eq:normalizedplap_strong1}. Finally, we
check that assumptions $\check u(x) = 0$ whenever $x \in \partial
B(y, 2r)$ and $\check u(x) = M$ whenever $x \in \partial B(y, r)$
imply $C = M/(e^{-\mu} - e^{-4\mu})$ and $D = M e^{-4\mu}/(e^{-\mu}
- e^{-4\mu})$. Let $A$ be as in the definition of
supersolution $\hat u$, see \eqref{eq:defnsuper} and  cf. the
discussion following \eqref{eq:normalizedplap_strong1}. Since $C =
A$, we obtain that bounds for $\log(|\nabla\check u|)$ are identical
to the case of supersolution. Therefore, the proof of the lemma is
completed.
\end{proof}

\subsection{Upper and lower $\px$-barriers of Bauman-type}

To find the upper bound we make use of the following barrier
functions.

\begin{lemma}\label{le:comparisonfunctions}
Let $y\in \R^n$ and $r>0$ be fixed and let $p$ be a bounded Lipschitz continuous
variable exponent on $\overline{B}(y,2r)$.
Let further $M > 0$ be given and for $x\in B(y, 2r)$ define functions
\begin{align*}
\hat u(x) = \frac{M}{1-2^{-\mu}} \left[1 - \left( \frac{r}{|x - y|}\right)^{\mu}\right]
\quad\hbox{ and }\quad \check u(x) = -\frac{2^{-\mu}M}{1-2^{-\mu}} \left[1 - \left( \frac{2r}{|x - y|}\right)^{\mu}\right]
\end{align*}
There exist $\mu_* = \frac{n-p^-+1}{p^--1}$ and $r_*$ depending only on $M, p^-,n$ and $\|\nabla p\|_{L^\infty}$ such that if $\mu \geq \mu_*$ and $r \leq r_*$, then $\hat u$ is a positive $\px$-supersolution while $\check u$ is a negative $\px$-subsolution in $B(y, 2r) \setminus B(y, r)$.
Moreover, $\hat u = M$ on $\partial B(y,2r)$ and $\hat u = 0$ on $\partial B(y, r)$, whereas $\check u = 0$ on $\partial B(y,2r)$ and $\check u = M$ on $\partial B(y, r)$.
\end{lemma}

\begin{proof}
Note that by the formal computations
$\div \left( |\nabla \hat u|^{p(x)-2} \nabla \hat u \right) \leq 0$
is equivalent to
\begin{align}
\left \langle \nabla \left(|\nabla \hat u|^{p(x)-2}\right), \nabla \hat u \right \rangle + |\nabla \hat u|^{p(x)-2} \Delta \hat u \leq 0. \label{eq:upper1}
\end{align}
Clearly $\hat u\in C^2(B(y, 2r) \setminus B(y, r))$ and thus we have that
\begin{align}
&\left \langle \nabla \left(|\nabla \hat u|^{p(x)-2}\right), \nabla \hat u \right \rangle
= \sum_{i=1}^n \frac{\partial}{\partial x_i} \left(|\nabla \hat u|^{p(x)-2}\right) \frac{\partial \hat u}{\partial x_i}
= \sum_{i=1}^n \frac{\partial}{\partial x_i} \left(e^{(p(x)-2)\log\left( |\nabla \hat u| \right)}\right) \frac{\partial \hat u}{\partial x_i}\nonumber \\
&= \sum_{i=1}^n \frac{\partial}{\partial x_i} \bigg\{(p(x)-2)\log\left( |\nabla \hat u| \right)\bigg\} |\nabla \hat u|^{p(x)-2} \frac{\partial \hat u}{\partial x_i}\nonumber \\
&= \sum_{i=1}^n \left\{ \frac{\partial p}{\partial x_i} \log\left( |\nabla \hat u| \right) + (p(x)-2)\frac{1}{|\nabla \hat u|} \frac{\partial}{\partial x_i} \left( |\nabla \hat u| \right) \right\} |\nabla \hat u|^{p(x)-2} \frac{\partial \hat u}{\partial x_i}\nonumber \\
&= \sum_{i=1}^n \left\{ \frac{\partial p}{\partial x_i} \log\left( |\nabla \hat u| \right) + (p(x)-2)\frac{1}{|\nabla \hat u|^2} \sum_{j=1}^n \frac{\partial \hat u}{\partial x_j}\frac{\partial^2 \hat u}{\partial x_j \partial x_i} \right\} |\nabla \hat u|^{p(x)-2} \frac{\partial \hat u}{\partial x_i}\nonumber \\
&=  \left\{ \langle \nabla p, \nabla \hat u \rangle \log\left( |\nabla \hat u| \right) + (p(x)-2)\frac{1}{|\nabla \hat u|^2} \Delta_{\infty} \hat u \right\} |\nabla \hat u|^{p(x)-2}.\label{eq:upper2}
\end{align}
Moreover, computations at \eqref{eq:u-hat-grad} (see below) will give us that in the given annulus $|\nabla \hat u| > 0$. This, together with \eqref{eq:upper1} and \eqref{eq:upper2} imply that we need the following inequality to be satisfied:
\begin{align}\label{eq:normalizedplap}
|\nabla \hat u|^2 \langle \nabla p, \nabla u \rangle \log\left( |\nabla \hat u| \right) + (p(x)-2) \Delta_{\infty} \hat u   + |\nabla \hat u|^2 \Delta \hat u \leq 0.
\end{align}
From \eqref{eq:normalizedplap} we see that comparing to the constant
$p$ case we have the extra term involving no second derivatives but
the gradient of both $\hat u$ and $\px$ instead.

Let us show first that $\hat u$ is a subsolution. This will be done by choosing
parameters $\mu$, $A$, $B$ and $r$ in the function
\begin{align*}
\hat u(x) = - A \left( \frac{r}{|x - y|}\right)^{\mu} + B, \qquad \textrm{where} \quad r < |x - y| < 2r.
\end{align*}
Differentiation of $\hat u$ yields
\begin{align*}
\hat u_{x_i} &= A \mu r^{\mu} |x-y|^{-(\mu + 2)} (x_i - y_i), \\
\hat u_{x_i x_j} &=  A \mu r^{\mu} |x-y|^{-(\mu + 4)} \bigg\{ |x-y|^2 \delta_{ij} - (\mu + 2)(x_i - y_i)(x_j - y_j)\bigg\}.
\end{align*}
Next we calculate the following expressions:
\begin{align}
| \nabla \hat u| &=  A \mu r^{\mu} |x-y|^{-(\mu + 1)},\label{eq:u-hat-grad}\\
\Delta \hat u &=  \sum_{i=1}^n \hat u_{x_i x_i}
= A \mu r^{\mu} |x-y|^{-(\mu + 4)}  \sum_{i=1}^n \bigg\{ |x-y|^2 \delta_{ii} - (\mu + 2)(x_i - y_i)^2 \bigg\} \nonumber \\
&=  A \mu r^{\mu} |x-y|^{-(\mu + 4)} \bigg\{ n |x-y|^2 - (\mu + 2)|x-y|^2 \bigg\} \nonumber \\
&=  A \mu r^{\mu} |x-y|^{-(\mu + 2)} \bigg\{ n  - \mu - 2 \bigg\}.\nonumber
\end{align}
As $\sum_{i,j = 1}^n \delta_{ij} (x_i - y_i)  (x_j - y_j) = |x-y|^2$ and $\sum_{i,j = 1}^n  (x_i - y_i)^2  (x_j - y_j)^2 = |x-y|^4$
we also get that
\begin{align*}
\Delta_{\infty} \hat u &=  \sum_{i,j = 1}^n \hat u_{x_i x_j} \hat u_{x_i} \hat u_{x_j}\\
&= A^3 \mu^3 r^{3\mu} |x-y|^{-(3\mu + 8)}  \sum_{i,j = 1}^n \bigg\{ |x-y|^2 \delta_{ij} - (\mu + 2)(x_i - y_i)(x_j - y_j)\bigg\}  (x_i - y_i)  (x_j - y_j)\\
&= A^3 \mu^3 r^{3\mu} |x-y|^{-(3\mu + 8)}  \bigg\{ |x-y|^4 - (\mu + 2)|x-y|^4\bigg\}\\
&= A^3 \mu^3 r^{3\mu} |x-y|^{-(3\mu + 4)}  \bigg\{1 - \mu - 2 \bigg\}.
\end{align*}
Next, we collect the above expressions and substitute them into \eqref{eq:normalizedplap}. After division by $|\nabla \hat u|^2$ we obtain the following inequality
\begin{align}
\langle \nabla p, \nabla \hat u \rangle \log\left( |\nabla \hat u| \right) + (p(x)-2) A \mu r^{\mu} |x-y|^{-(\mu + 2)}  \big\{1 - \mu - 2 \big\}
 +  A \mu r^{\mu} |x-y|^{-(\mu + 2)} \big\{ n  - \mu - 2 \big\} \leq 0.\label{eq:upper3}
\end{align}
Use $\langle \nabla p, \nabla u \rangle =  A \mu r^{\mu} |x-y|^{-(\mu + 2)} \langle \nabla p,  (x - y) \rangle$ in order to simplify \eqref{eq:upper3}:
\begin{align*}
\langle \nabla p,  (x - y) \rangle \log\left( |\nabla \hat u| \right) - \mu (\p - 1) + n - \p \leq 0.
\end{align*}
This holds true if
\begin{align}\label{eq:normalizedplap_strong}
 \|\nabla p\|_{L^\infty} |x - y| |\log\left( |\nabla \hat u| \right)| - \mu (p^- - 1) + n - p^- \leq 0.
\end{align}
We now chose $\mu_* = \mu_*(p^-, n) $ so that if $\mu \geq \mu_*$ we have
\begin{align}\label{eq:choseofmu}
 - \mu (p^- - 1) + n - p^- \leq -1.
\end{align}
In fact, $\mu_*=\frac{n-p^-+1}{p^--1}$.

Next we need function $\hat u$ to satisfy $|\hat u(x) - \hat u(z)| =
M$ whenever $x \in \partial B(y, 2 r)$ and $z \in \partial B(y, r)$.
This implies that $A= M/(1 -2^{-\mu})$. Our next step is to find conditions for $r$ so
that the first term on the left-hand side of
\eqref{eq:normalizedplap_strong} does not exceed value $-1$. Since
$|x-y|\leq 2r$ it is enough to ensure that
\begin{align}\label{eq:choseofalphar}
 \|\nabla p\|_{L^\infty} |\log\left( |\nabla \hat u| \right)| 2 r \leq 1.
\end{align}
Then the proof will be completed by \eqref{eq:normalizedplap_strong}, \eqref{eq:choseofmu} and \eqref{eq:choseofalphar}. Hence it only remains to satisfy \eqref{eq:choseofalphar}.
We have
\begin{align*}
| \nabla \hat u| =  \frac{M}{1 -2^{-\mu}} \mu r^{\mu} |x-y|^{-(\mu +
1)} =  \frac{M}{1 -2^{-\mu}} \mu r^{-1} \left(
\frac{r}{|x-y|}\right)^{\mu + 1}.
\end{align*}
Since $r < |x-y| < 2 r$, it holds:
\begin{align*}
1<\frac{M}{1 -2^{-\mu}} \frac{\mu} r \frac{1}{2^{\mu + 1}} \leq | \nabla \hat u| \leq  \frac{M}{1 -2^{-\mu}} \frac{\mu} r.
\end{align*}
We choose $r$ so small that the left hand side is larger than one.
Such a requirement leads to condition that $r <
r_{**}:=\frac{M\mu}{2(2^\mu-1)}$ and thus $r_{**}$ depends on $M$
and $\mu_*$ and therefore on $M$, $p^-$ and $n$. Now $|\log\left(
|\nabla \hat u| \right)| \leq |\log\left( M \mu /
(1-2^{-\mu})\right)  - \log(r)| \leq |\log\left( M \mu /
(1-2^{-\mu})\right)|  + |\log(r)| $. As $\lim_{r\to 0^+}r|\log(r)|
\to 0$ we have
\begin{align}
\|\nabla p\|_{L^\infty} |\log\left( |\nabla \hat u| \right)| 2 r  \leq  \|\nabla p\|_{L^\infty}  2 r  \bigg\{ |\log\left( M \mu / (1-2^{-\mu})\right)|  + |\log(r)|\bigg\} \leq 1,\label{eq:upper4}
\end{align}
provided that $r \leq r_*$ is small enough. Indeed, if
$r|\log r|<1/(2\|\nabla p\|_{L^\infty})$ and $r<1/(4|\nabla p\|_{L^\infty}|\log (2^{\mu+1}r_{**})|)$, then \eqref{eq:upper4}
holds. In a consequence $r_*$ depends only on $M$, $\|\nabla
p\|_{L^\infty}$, $p^{-}$ and $n$. The last inequality completes the
proof of \eqref{eq:choseofalphar}. By taking $B:=-A$ it holds that
$\hat u$ satisfies the boundary value conditions.

 In order to show that $\check u$ is a $\px$-subsolution we proceed in the analogous way as in the second part of Lemma~\ref{le:comparisonfunctions_W}, cf. discussion between formulas \eqref{eq:subsol} and \eqref{eq:normalizedplap_strong1ny}.
 We define
\begin{align*}
\check u(x) = A \left( \frac{r}{|x - y|}\right)^{\mu} + B, \qquad \textrm{where} \quad r < |x - y| < 2r.
\end{align*}
 Similarly to computations for $\hat u$ we obtain that $\check u_{x_i}=-\hat u_{x_i}$, $\check u_{x_ix_j}=-\hat u_{x_ix_j}$. Hence $\Delta \check u=-\Delta \hat u$ and $\Delta_{\infty} \check u=-\Delta_{\infty} \hat u$. Upon collecting these expressions we use them in \eqref{eq:upper2} together with $\div \left( |\nabla \check u|^{p(x)-2} \nabla \check u \right) \geq 0$. In a consequence we arrive at the following inequality:
 \begin{align}
-\langle \nabla p, \nabla \check u \rangle \log\left( |\nabla \hat u| \right) - (p(x)-2) A \mu r^{\mu} |x-y|^{-(\mu + 2)}  \big\{1 - \mu - 2 \big\}
 -  A \mu r^{\mu} |x-y|^{-(\mu + 2)} \big\{ n  - \mu - 2 \big\} \geq 0.
\end{align}
This condition is the same as \eqref{eq:upper3}. Finally, we check
that assumptions $\check u(x) = 0$ whenever $x \in \partial B(y, r)$
and $\check u(x) = M$ whenever $x \in \partial B(y, 2r)$ imply $C =
-A$ and $D = -2^{-\mu}B$, where $A$ and $B$ are as in the definition
of supersolution $\hat u$. From these we infer that $|\nabla \check
u|=|\nabla \hat u|$ and hence the bounds for $\log(|\nabla\check
u|)$ are the same as in the case of supersolution. Thus, the proof
for $\px$-subsolutions, and for Lemma~\ref{le:comparisonfunctions},
is completed.
\end{proof}

\section{Upper and lower boundary growth estimates. The boundary Harnack inequality}\label{Sect-estimates}

\setcounter{theorem}{0}
\setcounter{equation}{0}

This section contains main result of the paper, namely the proof of the boundary Harnack inequality for positive $\px$-harmonic functions on domains satisfying the ball condition, see Theorem~\ref{thm:harnack}. The proof relies on Lemmas~\ref{le:lower} and \ref{le:upper}, where we show the lower- and, respectively, the upper estimates for a rate of decay of a $\px$-harmonic function close to a boundary of the underlying domain.
In particular, Lemmas~\ref{le:lower} and \ref{le:upper} imply stronger result than the usual boundary Harnack inequality. Namely, that a $\px$-harmonic function vanishes at the same rate as the distance function. Moreover, Lemma~\ref{le:lower} illustrates the following phenomenon: the geometry of the domain effects the sets of parameters on which the rate of decay depends.
Indeed, it turns out that constants in our lower estimate depend whether domain satisfies the interior ball condition
or the ball condition, cf. parts (i) and (ii) of Lemma~\ref{le:lower}.
As a corollary we also obtain a decay estimate for supersolutions
(a counterpart of Proposition~6.1 in Aikawa--Kilpel\"ainen--Shanmugalingam--Zhong~~\cite{AKSZ07}).

For $w\in \partial \Om$ we denote by $A_r(w)$ a point satisfying
$d(A_{r}(w), \partial \Omega) = r$ and $|A_{r}(w) - w| = r$.
Existence of such a point is guaranteed by the interior ball condition (with radius $r_i$) for $r \leq r_i/2$.
Recall also that by $c_H$ we denote the constant from the Harnack inequality, Lemma~\ref{harnack}.

\begin{lemma}[Lower estimates]\label{le:lower}
Let $\Omega \subset \R^n$ be a domain satisfying the interior ball
condition with radius $r_i$, $w \in \partial \Omega$ and $0 < r <
r_i$. Let $p$ be a bounded Lipschitz continuous variable exponent.
Assume that $u$ is a positive $\px$-harmonic function in
$\Omega \cap B(w, r)$ satisfying $u = 0$ on $\partial \Omega \cap
B(w,r)$. Then the following is true.

\begin{itemize}
\item[(i)]
There exist constants $c$ and $\tilde c$ such that if $\tilde r: = r / \tilde c$ then
\begin{align*}
c\, u(x) \geq  \frac{d(x, \bdry\Om)}{r} \qquad \text{for} \quad x
\in \Omega \cap B(w, \tilde r).
\end{align*}

The constant $c$ depends on $\inf_{\Gamma_{w,\tilde
r}} u, r_i$ and $p^+, p^-, n, \|\nabla p\|_{L^\infty}$, where $\Gamma_{w,\tilde
r} = \{x \in \Omega | \tilde r < d(x, \bdry\Om) < 3 \tilde r \} \cap
B(w,r)$. The constant $\tilde c$ depends only on $r_i$ and $p^-, \|\nabla
p\|_{L^\infty}$.
\end{itemize}
\noindent
Assume in addition that $\Omega$ satisfies the ball condition with radius $r_b$ and that $0 < r < r_b$.

\begin{itemize}
\item[(ii)]
Then there exist constants $c_L$ and $\tilde c_L$ such that if $\tilde r: = r / \tilde c_L$ then
\begin{align*}
c_L \, u(x) \geq \frac{d(x, \bdry\Om)}{r} \qquad \text{for} \quad x
\in \Omega \cap B(w, \tilde r).
\end{align*}

The constant $c_L$ depends on $\sup_{\Om\cap
B(w, r)} u, u(A_{2\tilde r}(w)), r_b$ and
$p^+, p^-, n, \|\nabla p\|_{L^\infty}$, while $\tilde c_L$ depends only on $r_b$
and $p^-, \|\nabla p\|_{L^\infty}$.
\end{itemize}

\end{lemma}

\begin{proof}
To prove $(i)$, we start by applying Lemma \ref{le:comparisonfunctions_W} to obtain $r_*$,
depending only on $\|\nabla p\|_{L^\infty}, p^{-}$,
such that we can construct barriers in an annulus with radius less than $r_*$.
Assume $\tilde c$ to be so large that $\tilde r \leq \min \{ r_*, r/6 \}$ and note that so far $\tilde c \geq 6$
depends only on $\|\nabla p\|_{L^\infty}, p^{-}$ and $r_i$.

Let $x \in \Omega \cap B(w, \tilde r)$ be arbitrary. Then there
exists $\eta \in \partial \Omega$ such that $d(x, \bdry\Om) = |x -
\eta|$. By the interior ball condition at $\eta$ we find a point
$\eta^i$ such that $B(\eta^i, r_i) \subset \Omega$ and $\eta \in
\partial B(\eta^i, r_i)$. Take $\eta^i_{2 \tilde r} \in [\eta,
\eta^i]$ with $d(\eta^i_{2 \tilde r}, \bdry \Om) = 2 \tilde r$.
Since $\tilde r \leq r/6$ we have $B(\eta^i_{2 \tilde r}, 2\tilde r)
\subset \Omega \cap B(w, r)$. Thus, $u$ is a positive $\px$-harmonic
function in $B(\eta^i_{2\tilde r}, 2\tilde r)$.
Next we note that $B(\eta^i_{2\tilde r}, \tilde r) \subset \Gamma_{w,\tilde r}$ and since $u$ is continuous
\begin{align}\label{eq:note1}
u \geq \inf_{\Gamma_{w,\tilde r}} u > 0\qquad\hbox{in}\quad \Gamma_{w,\tilde r}.
\end{align}
Using \eqref{eq:note1} and Lemma \ref{le:comparisonfunctions_W} we
construct a subsolution $\check u$ in $B(\eta^i_{2\tilde r}, 2\tilde
r) \setminus B(\eta^i_{2\tilde r}, \tilde r)$ with boundary values
$\check u \equiv \inf_{\Gamma_{w,\tilde r}} u$ on $\partial
B(\eta^i_{2\tilde r}, \tilde r)$ and $\check u \equiv 0$ on
$\partial B(\eta^i_{2\tilde r}, 2\tilde r)$. Since $\check u \leq u$
on $\partial B(\eta^i_{2\tilde r}, \tilde r)$ and $0 = \check u \leq
u$ on $\partial B(\eta^i_{2\tilde r}, 2\tilde r)$, we obtain that
$\check u \leq u$ in $B(\eta^i_{2\tilde r}, 2\tilde r) \setminus
B(\eta^i_{2\tilde r}, \tilde r)$ by the comparison principle
(Lemma~\ref{comparison_princ}). By the above discussion $x \in
B(\eta^i_{2\tilde r}, 2\tilde r) \setminus B(\eta^i_{2\tilde r},
\tilde r)$ and the result will follow by showing that $\check u$
does not vanish faster than $d(x, \bdry\Om)$ as $x \to \partial
\Omega$. In order to show this we observe that the derivative of
$\check u$ in a direction normal to $\partial B(\eta^i_{2\tilde r},
2\tilde r)$ does not vanish. Indeed, using $\mu = \mu_*$ in Lemma
\ref{le:comparisonfunctions_W}, $\tilde r \leq |x-\eta^i_{2\tilde
r}| \leq 2\tilde r$ together with computations for $\nabla \check u$
in \eqref{eq:subsol} results in the following estimate:
\begin{equation}
\left|\left\langle \nabla \check u (x), \frac{x-\eta^i_{2\tilde r}}{|x-\eta^i_{2\tilde r}|}\right\rangle\right|
= \frac{2 \mu_*  \inf_{\Gamma_{w,\tilde r}} u }{ \tilde r^2} \frac{e^{-\mu_* \left( \frac{|x - \eta^i_{2\tilde r}|}{\tilde r}\right)^2 }}{e^{-\mu_*} - e^{-4\mu_*}} |x - \eta^i_{2\tilde r}| \geq \frac{2 \tilde c \mu_*  \inf_{\Gamma_{w,\tilde r}} u}{r}\frac{e^{-3\mu_*}}{1 - e^{-3\mu_*}} \geq  \frac1{c r}.
\label{eq:lower-normalder1}
\end{equation}
Since $\tilde c$ depends only on $r_i, \|\nabla p\|_{L^\infty}, p^{-}$,
while $\mu_*$ may depend on $ \inf_{\Gamma_{w,\tilde r}} u, \|\nabla p\|_{L^\infty}, p^{-}, p^{+}, n$,
the constant $c$ depends only on $ \inf_{\Gamma_{w,\tilde r}} u, r_i, \|\nabla p\|_{L^\infty}, p^{-}, p^{+}, n$.
Inequality \eqref{eq:lower-normalder1} completes the proof of $(i)$ in Lemma~\ref{le:lower}.

To prove $(ii)$, assume $\tilde c_L$ to be so large that $\tilde r = r/\tilde c_L  \leq \min \{ r_*, r/6 \}$ and note that now $\tilde c_L$ depends only on $\|\nabla p\|_{L^\infty}, p^{-}$ and $r_b$.
We proceed as in the former case but with \eqref{eq:note1} replaced by the following claim:
\begin{align}\label{eq:claim1}
u \geq \frac{1}{c_0} \qquad \textrm{on} \quad
\overline{B}(\eta^i_{2\tilde r}, \tilde r),
\end{align}
where $c_0$ depends only on $c_H, u(A_{2\tilde r}(w)), \|\nabla
p\|_{L^\infty}, r_b, p^{-}, p^{+}, n$, see details below and
Figure~\ref{fig:geometry}. In a consequence we obtain, instead of
\eqref{eq:lower-normalder1}, the following inequality:
\begin{equation}
\left|\left\langle \nabla \check u (x), \frac{x-\eta^i_{2\tilde r}}{|x-\eta^i_{2\tilde r}|}\right\rangle\right|
= \frac{2 \mu_*}{c_0 \tilde r^2} \frac{e^{-\mu_* \left( \frac{|x - \eta^i_{2\tilde r}|}{\tilde r}\right)^2 }}{e^{-\mu_*} - e^{-4\mu_*}} |x - \eta^i_{2\tilde r}| \geq \frac{2 \tilde c_L \mu_*}{ c_0  r}\frac{e^{-3\mu_*}}{1 - e^{-3\mu_*}} \geq  \frac1{c_L r}.
\label{eq:lower-normalder}
\end{equation}
Since $\tilde c_L$ depends only on $\|\nabla p\|_{L^\infty}, r_b, p^{-}$ and $c_0$, $\mu_*$ depend only on $c_H, u(A_{2\tilde r}(w)), r_b, \|\nabla p\|_{L^\infty}, p^{-}, p^{+}, n$,
inequality \eqref{eq:lower-normalder} completes the proof of Lemma~\ref{le:lower} under the assumed claim \eqref{eq:claim1}.

To prove claim \eqref{eq:claim1} we proceed as follows.
By using Harnack's inequality in $B(A_{2\tilde r}(w),r_H)$, for some $r_H$ to be chosen later,
we obtain (cf. Lemma~\ref{harnack})
\begin{align*}
\sup_{B(A_{2\tilde r}(w), r_H)} u \leq c_H \Big( \inf_{B(A_{2\tilde
r}(w), r_H)} u + r_H \Big)
\end{align*}
and so
$\frac{1}{c_H} u(A_{2\tilde r}(w)) - r_H \leq u(x)$ in $B(A_{2\tilde r}(w), r_H)$.
Assume $r_H$ so small that $r_H \leq \min\{\frac{1}{2 c_H} u(A_{2\tilde r}(w)), \tilde r\}$
and observe that then
\begin{align*}
\frac{1}{2 c_H} u(A_{2\tilde r}(w)) \leq u(x) \qquad \text{in} \quad B(A_{2\tilde r}(w), r_H).
\end{align*}
We now use Lemma \ref{le:comparisonfunctions_W} to find a subsolution $\check u$ in
$B(A_{2\tilde r}(w), 2 r_H) \setminus B(A_{2\tilde r}(w), r_H)$ satisfying $\check u \equiv \frac{1}{2 c_H} u(A_{2\tilde r}(w))$ on $\partial B(A_{2\tilde r}(w), r_H)$ and $\check u\equiv 0$ on $\partial B(A_{2\tilde r}(w), 2 r_H)$.
 The definition of $A_{2\tilde r}(w)$ and $r_H \leq \tilde r$ give us that $B(A_{2\tilde r}(w), 2 r_H) \subset \Omega\cap B(w,r)$.
By the comparison principle we obtain that $u\geq \check u$ in $B(A_{2\tilde r}(w), 2 r_H) \setminus  B(A_{2\tilde r}(w), r_H)$.
In particular, $ \frac{1}{2 c_H} u(A_{2\tilde r}(w)) \leq u$ and
\begin{align*}
\frac{1}{c_1}  \leq u(x) \qquad \text{in} \quad B(A_{2\tilde r}(w), 3/2\, r_H).
\end{align*}
Constant $c_1$ arises from computing $\check u$ for $x$ such that $|x-A_{2\tilde r}(w)|=\frac{3}{2}r_H$ (cf. the definition of $\check u$ in Lemma~\ref{le:comparisonfunctions_W}).
Furthermore, $c_1 > 2 c_H$ depends only on $c_H, u(A_{2\ti r}(w)), \|\nabla p\|_{L^\infty}, p^{-}, p^{+}, n$,
since $\mu_*$ in Lemma \ref{le:comparisonfunctions_W} depends only on these parameters.

We proceed by constructing a sequence of barrier functions and building a chain of balls joining points $A_{2\ti r}(w)$ and $\eta^i_{2\tilde r}$, where $\eta^i_{2\tilde r}$ is the same point as discussed in part (i) of the proof.
Using the ball condition we find that if $\tilde r$ is small enough, depending only on $r_b$, then
\begin{align*}
d([\eta^i_{2\tilde r}, A_{2\tilde r}(w)], \partial \Omega) \geq {\tilde r}.
\end{align*}
That such $\tilde r$ can be found follows from the argument similar
to the one presented in Sections 2 and 3 in
Aikawa--Kilpel\"ainen--Shanmugalingam--Zhong~\cite{AKSZ07} as
$\Omega$ is a $C^{1,1}$-domain, and thus the unit normal is
Lipschitz continuous.

Consider the subsolution in $B(y, 2r_H) \setminus B(y, r_H)$ for a $y \in [\eta^i_{2\tilde r}, A_{2\tilde r}(w)]$
with boundary values $\frac{1}{c_1}$ on $B(y, r_H)$ and $0$ on $B(y, 2 r_H)$.
Put $y$ as close as possible to point $\eta^i_{2\tilde r}$ under the restriction that
$B(y, r_H) \subset B(A_{2\tilde r}(w), 3/2 \,r_H)$.
By the comparison principle we then obtain that
\begin{align*}
\frac{1}{c_2}  \leq u(x) \qquad \text{in} \quad B(A_{2\tilde r}(w), 3/2\; r_H) \cup B(y, 3/2\; r_H)
\end{align*}
where $c_2 > c_1 > 2 c_H$ depends only on $c_H, u(A_{2\ti r}(w)), \|\nabla p\|_{L^\infty}, p^{-}, p^{+}, n$.
Proceeding in this way we obtain a chain of balls centered at points $y$ which, eventually, contain $\eta^i_{2\tilde r}$,
see Figure \ref{fig:geometry}.
Indeed, each ball adds distance $r_H/2$ to the length of chain,
and hence the number of balls needed to approach $\eta^i_{2\tilde r}$ depends only on $\tilde r / r_H$,
which in turn depends only on $c_H, u(A_{2\ti r}(w)), \|\nabla p\|_{L^\infty}, p^{-}$ and $r_b$.
We can proceed in the same way to cover $B(\eta^i_{2\tilde{r}}, \tilde r)$.
Hence we conclude the proof of \eqref{eq:claim1} and therefore the proof of Lemma \ref{le:lower}.
\begin{figure}[!hbt]
\begin{center}
\includegraphics[width=10.5cm,height=8cm,viewport=65 20 330 230,clip]{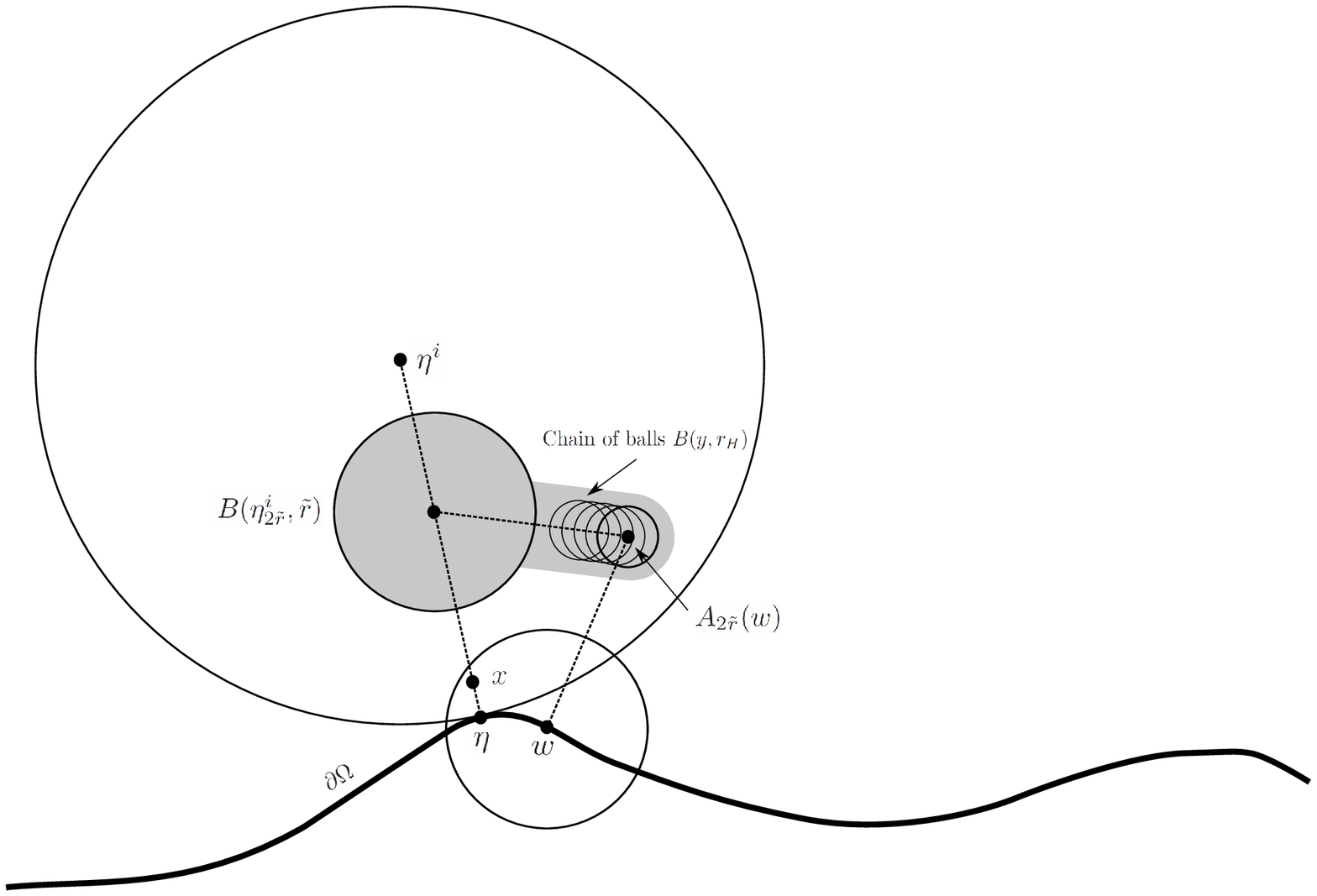}
\end{center}
\caption{The geometry in the proof of Claim \eqref{eq:claim1}.
The chain of annuli $B(y, 2r_H)\setminus B(y,r_H)$ covers the grey-shaded area.
Using this chain and associated subsolutions from Lemma \ref{le:comparisonfunctions_W}
we prove that $u \geq c_0^{-1}$ on $B(\eta^i_{2\tilde r}, \tilde r)$.}
\label{fig:geometry}
\end{figure}
\end{proof}

Denote $u^*$ the lsc-regularization of a supersolution $u$ (see e.g. Adamowicz-Bj\"orn-Bj\"orn~\cite[Theorem 3.5]{ABB13} and discussion therein).

\begin{cor}[cf. Proposition 6.1 in \cite{AKSZ07}]\label{cor:lower}
Let $\Om \subset \R^n$ be a bounded domain satisfying the interior
ball condition for $r_i$ and let $p$ be a bounded Lipschitz continuous variable exponent such that $p^{+}<n$. Furthermore, let
$u\geq 0$ be a supersolution in $\Om$. If there exists a point $w\in
\partial \Om$ such that
\begin{equation*}
 \liminf_{\Omega \ni y\to w} \frac{u^*(y)}{d(y, \bdry\Om)}\,=\,0,
\end{equation*}

then $u^*\equiv 0$ in $\Om$.
\end{cor}

\begin{proof}
Suppose that $u^*\geq 0$ is a lsc-regularization of supersolution $u$ in $\Om$ such that $u^* \not\equiv 0$.
Since the Lipschitz continuity assumption on $\px$ implies the Dini type condition (see (5.1) in H\"ast\"o--Harjulehto--Latvala--Toivanen~\cite{HHLT}),
we can apply the strong minimum principle (see Theorem 5.3 in \cite{HHLT}) to obtain $u^*>0$ in $\Om$.
Hence
\begin{equation*}
m := \inf_{w \in \partial\Omega} \bigg( \inf_{\Gamma_{w,\tilde r}} u^*(x) \bigg) > 0,
\end{equation*}
where $\Gamma_{w,\tilde r}$ is as in Lemma \ref{le:lower} (i).
Clearly, the infimum is finite as well. Indeed, by Theorem~3.5 in
Adamowicz--Bj\"orn--Bj\"orn~\cite{ABB13} we have that $u^*$ is a
(quasicontinuous) supersolution. Furthermore, Corollary~4.7 in
Harjulehto--Kinnunen--Lukkari~\cite{HKL07} implies that the
$\px$-capacity of the polar set of $u^*$ is zero, i.e.
$C_{\px}(\{u^*=\infty\})=0$, see Definition~\ref{C-px-defn}, also
\cite{ABB13} and \cite{HKL07} for further discussion. Thus
$m<\infty$.

Since the comparison principle applies to supersolutions, we proceed as in the proof of Lemma \ref{le:lower} part $(i)$ to obtain
\begin{equation*}
 \frac{u^*(y)}{d(y, \bdry\Om)} \geq \frac{1}{c r} >0 \qquad \hbox{for all $x$ close enough to $\partial \Omega$},
\end{equation*}
where $c$ depends only on $m, r_i, \|\nabla p\|_{L^\infty}, p^{-}, p^{+}, n$.
Thus
\[
  \liminf_{\Omega \ni y\to w} \frac{u^*(y)}{d(x, \bdry\Om)}\,>\,0
\]
for all $w \in \partial \Omega$ and the corollary is proven.
\end{proof}

We now show the upper boundary growth estimates.

\begin{lemma}[Upper estimates]\label{le:upper}
Let $\Omega \subset \R^n$ be a domain satisfying the exterior ball
condition with radius $r_e$, $w \in \partial \Omega$ and $0 < r <
r_e$. Let $p$ be a bounded Lipschitz continuous
variable exponent.
Assume that $u$ is a positive $\px$-harmonic function in
$\Omega \cap B(w, r)$ satisfying $u = 0$ on $\partial \Omega \cap
B(w,r)$. Then there exist constants $c$ and $\tilde c$ such that if $\tilde r = r / \tilde c$ then
\begin{align*}
 u(x) \leq c \, \frac{d(x, \bdry\Om)}{r} \qquad \text{for} \quad x \in \Omega \cap B(w, \tilde r).
\end{align*}
The constant $c$ depends on
$\sup_{B(w,2 \ti r)\cap \Om} u, r_e$ and
$n, p^+, p^-, \|\nabla p\|_{L^\infty}$,
while $\tilde c$ depends on $r_e$ and $p^-, \|\nabla p\|_{L^\infty}$.
\end{lemma}

\begin{proof}
 We first apply Lemma \ref{le:comparisonfunctions_W} to obtain radius $r_*$, depending
only on $\|\nabla p\|_{L^\infty}$ and $p^{-}$, such that we can
construct barriers in annulus with radius less than $r_*$. Assume
$\tilde c$ to be so large that $\tilde r \leq \min \{ r_*, r/5 \}$
and note that so far $\tilde c \geq 5$ depends only on $\|\nabla
p\|_{L^\infty}, p^{-}$ and $r_e$. Let $x \in \Omega \cap B(w, \tilde
r)$ be arbitrary. Then there exists $\eta \in \partial \Omega$ such
that $d(x, \bdry\Om) = |x - \eta|$. By the exterior ball condition
at $\eta$ we find a point $\eta^e$ such that $B(\eta^e, r_e) \subset
\R^n \setminus\Omega$ and $\eta \in \partial B(\eta^e, r_e)$. Take
$\eta^e_{ \tilde r} \in [\eta, \eta^e]$ with $d(\eta^e_{ \tilde r},
\bdry \Om) = \tilde r$. Since $\tilde r = r/5$ we have $B(\eta^e_{
\tilde r}, 2\tilde r)\cap\Omega \subset \Omega \cap B(w, r)$. We now
use Lemma~\ref{le:comparisonfunctions_W} with $M := \sup_{B(w,2\ti
r)\cap \Om} u$ to obtain a $\px$-supersolution $\hat u$ in the
annulus $B(\eta^e_{ \tilde r}, 2\tilde r)\setminus B(\eta^e_{ \tilde
r}, \tilde r)$ satisfying $\hat u \equiv 0$ on $\partial B(\eta^e_{
\tilde r}, \tilde r)$ and $\hat u \equiv \sup_{B(w,2\ti r)\cap \Om}
u$ on $\partial B(\eta^e_{ \tilde r}, 2\tilde r)$. By the comparison
principle in Lemma~\ref{comparison_princ}, we obtain $u \leq \hat u$
in $ \Omega \cap B(\eta^e_{ \tilde r}, 2\tilde r)$ and since $x$ is
in this set the result will follow by showing that $\hat u$ vanishes
at least as fast as $d(x, \bdry\Om)$ when $x \to \partial \Omega$.
Indeed, putting $\mu=\mu^*$, $\tilde r \leq |x-\eta^e_{\tilde r}|
\leq 2\tilde r$ together with computations for $\nabla \hat u$ in
\eqref{eq:supergrad} results in the following estimate:
\begin{equation}
\left|\left\langle \nabla \hat u (x), \frac{x-\eta^e_{\tilde r}}{|x-\eta^e_{\tilde r}|}\right\rangle\right|
= \frac{2 \mu_* \sup_{B(w,2\ti r)\cap \Om} u}{ \tilde r^2} \frac{e^{-\mu_* \left( \frac{|x - \eta^e_{\tilde r}|}{\tilde r}\right)^2 }}{e^{-\mu_*} - e^{-4\mu_*}} |x - \eta^e_{\tilde r}| \leq
\frac{4 \tilde c \mu_* \sup_{B(w,2 \ti r)\cap \Om} u}{r}\frac{1}{1 - e^{-3\mu_*}} \leq
\frac{c}{r}.
\label{eq:lower-normalder_upp}
\end{equation}
Since $\mu_*$ and $\tilde c$ bring dependence on $\sup_{B(w,2 \ti r)\cap \Om} u, r_e$ and
$\|\nabla p\|_{L^\infty}, p^{-}, p^{+}, n$,
we conclude that $c$ depends on the same set of parameters and thus inequality \eqref{eq:lower-normalder_upp}
completes the proof of 
Lemma~\ref{le:upper}.
\end{proof}

We are now in a position to state and prove the main result of the paper.
\begin{theorem}[Boundary Harnack inequality]\label{thm:harnack}
Let $\Omega \subset \R^n$ be a domain satisfying the ball condition
with radius $r_b$. Let $w \in \partial \Omega$, $0 < r < r_b$ and let $p$ be a bounded Lipschitz continuous variable exponent.
Assume that $u$ and $v$ are positive $\px$-harmonic functions in $\Omega \cap
B(w, r)$, satisfying $u = 0 = v$ on $\partial \Omega \cap B(w,r)$.
Then there exist constants $c$, $C$ and $\tilde c$ such that
\begin{align*}
(i)\quad \frac{1}{c} \frac{d(x, \bdry\Om)}{r} \leq u(x) \leq c
\frac{d(x, \bdry\Om)}{r}, \qquad (ii) \quad \frac{1}{C} \leq
\frac{u(x)}{v(x)} \leq C \qquad \text{for} \quad x \in \Omega \cap
B(w, r/\ti c).
\end{align*}
 The constant $\tilde c$ depends on $r_b$ and $p^-, ||\nabla p||_{L^{\infty}}$, constant $c$ depends on $n, p^+, p^-, ||\nabla p||_{L^{\infty}}$, $\sup_{B(w, r)\cap \Om} u, u(A_{2\tilde r}(w))$ and $r_b$,
while $C$ depends on the same parameters as $c$ and also on $v(A_{2\tilde r}(w))$ and $\sup_{B(w, r)\cap \Om} v$.
\end{theorem}

\begin{rem}
 Conclusion $(ii)$ in Theorem \ref{thm:harnack} is sometimes formulated in the following form
(also referred to as a boundary Harnack inequality):
For all points $x, y \in \Omega \cap B(w, r/\ti c)$ it holds that
 \[
   \frac{1}{{C}^2}\leq \frac{u(x)}{v(x)} \frac{v(y)}{u(y)}\leq {C}^2,
 \]
where $\tilde c$ and $C$ are the constants from Theorem \ref{thm:harnack}.
\end{rem}

\begin{proof}[Proof of Theorem~\ref{thm:harnack}]
 We observe that $\sup_{B(w,2 \ti r)\cap \Om} u$ in Lemma
\ref{le:upper} can be replaced by $\sup_{B(w, r)\cap \Om} u$. Then
Lemma \ref{le:lower} and Lemma \ref{le:upper} immediately imply the
assertion of the theorem.
\end{proof}

\section{$\px$-Harmonic measure}\label{Sect-measures}

\setcounter{theorem}{0}
\setcounter{equation}{0}

In this section we study $\px$-harmonic measures. In Lemma~\ref{p(x)measure_exists} we show the existence of a
$\px$-harmonic measure and in Theorem \ref{p(x)measure} we provide our main results of this section: lower- and upper- growth estimates for such measures. Finally, using these growth estimates and the Carleson estimate (Theorem \ref{Carleson}) we conclude in Corollary \ref{cor:doubling}
a weak doubling property of the $\px$-harmonic measure.
Let us now explain motivations for our studies.

Harmonic measures were employed to prove a Boundary Harnack
inequality in the setting of harmonic functions, see Dahlberg
\cite{D77} and Jerison--Kenig \cite{JK82}. When studying boundary
behavior of $p$-harmonic type functions, various versions of
generalizations of harmonic measures have been introduced and
studied for $p \not = 2$, see e.g.
Llorente--Manfredi--Wu~\cite{lmw}. In the case of constant $p$ ($p
\not = 2$) Bennewitz and Lewis employed the doubling property of a
$p$-harmonic measure, first proved in Eremenko--Lewis~\cite{EL91},
to obtain a Boundary Harnack inequality for $p$-harmonic functions
in the plane, see Bennewitz--Lewis~\cite{BL}. This result has been
generalized to the setting of Aronsson-type equations by
Lewis--Nystr\"om~\cite{LN08} and Lundstr\"om--Nystr\"om~\cite{LN11}.
The $p$-harmonic measure, defined as in the aforementioned papers,
as well as Boundary Harnack inequalities, have played a significant role when
studying free boundary problems, see for example Lewis--Nystr\"om
\cite{LN12}. The $p$-harmonic measure was also used to find the
optimal H\"older exponent of $p$-harmonic functions vanishing near
the boundary, see Kilpel\"ainen--Zhong~\cite{KZ01} and
Lundstr\"om~\cite{N11}.
Moreover, a work of Peres--Sheffield~\cite{ps} provides discussion of connections
between $p$-harmonic measures, defined in a different way though, and tug-of-war games.
As for the equations with nonstandard growth we mention paper by Lukkari--Maeda--Marola~\cite{lmm}, where some upper estimates for $\px$-harmonic measures were studied in the context of Wolff potentials.

To prove our results concerning $\px$-measures we begin by stating a Caccioppoli-type estimate.

\begin{lemma}[Caccioppoli-type estimate]\label{energy}
Let $\Om\subset\R^{n}$, $p$ be a bounded log-H\"older continuous variable exponent and assume that $\eta \in C_0^{\infty}(\Om)$ with $0 \leq \eta \leq 1$.
If $u$ is a $\px$-subsolution in $\Omega$, then
\begin{equation*}
\int_{\Om}|\nabla u|^{\p} \eta^{p^+} dx \leq c\, \int_{\Om}|u|^{\p} |\nabla \eta|^{\p} dx,
\end{equation*}
where $c=c(p^+)$.
\end{lemma}

\begin{proof}
The proof goes the same lines as for the $p=const$ case, namely one uses in \eqref{eq-sub-sol} a test function $\phi=u\eta^{p^+}$, cf. Lemma~5.3 in Harjulehto--H\"ast\"o--Latvala~\cite{HHL08} for the proof of Caccioppoli estimate in the case of slightly modified $\px$-Laplace operator $\div(\px|\nabla u|^{\px-2}\nabla u)$).
\end{proof}

The following existence lemma is probably known to experts in the variable exponent analysis,
but to our best knowledge have not appeared earlier in the literature.
Therefore, we include its proof for the readers convenience.

\begin{lemma}\label{p(x)measure_exists}
Assume that $\Omega\subset\R^{n}$, $w \in\partial\Omega$, $0 < r <
\infty$ and let $p$ be a bounded log-H\"older continuous variable exponent.
Suppose that $u$ is a positive $\px$-harmonic function in
$\Omega \cap B(w,2r)$, continuous on $\bar \Omega \cap B(w,2r)$ with
$u \equiv 0$ on $\partial \Omega \cap B(w,2r)$. Extend $u$ to $B(w, 2r)$ by defining
$u \equiv 0$ on $B(w, 2r)\backslash \Omega$.
Then there exists a unique finite positive Borel measure $\mu$ on $\R^n$,
with support in $\partial \Omega \cap B(w,r)$, such that
whenever $\psi \in C_0^{\infty}(B(w,r))$ then
\begin{align}\label{p(x)measure_exists:assert}
\int \limits_{\R^n} \langle  |\nabla u|^{p(x)-2} \nabla u, \nabla
\psi \rangle dx = -  \int \limits_{\R^n} \psi d \mu.
\end{align}
\end{lemma}

\begin{proof}
We first prove that the extended function is a subsolution in $B(w,r)$.
To do so, we begin by showing that the extension, denoted
by $U$, belongs to $W^{1, \px}_{loc}(B(w,r))$.
It is immediate that $U$ belongs to $L^{\px}(B(w, r))$ and that $\nabla U\in
L^{\px}(B(w, r))$. To conclude that $U\in W^{1, \px}_{loc}(B(w,
r))$ it remains to show that $U\in W^{1, \px}_{loc}(B_R)$ for any
ball $B_R \Subset B(w,r)$, which in turn boils down to showing that
$\nabla u$ is the distributional gradient of $U$ in $B_R$. Indeed, let $\eta\in C^\infty_0(B_R)$ be arbitrary and let $\phi \in C^\infty_0((\Om \cap B_R)\cup \rm{supp}\,\eta)$ be such
that $0\leq \phi \leq 1$ and $\phi\equiv 1$ on $\Om\cap B_R\cap
\rm{supp}\, \eta$. Then $\eta \phi \in C_{0}^{\infty}((\Om\cap B_R)\cup \rm{supp}\,\eta)$. Since $\nabla u$ is the distributional gradient of $u$ in
$\Om\cap B(w, r)$ and $u\equiv 0$ in $B_R\setminus \Om$, we have
\begin{align*}
0&=\int_{(\Om\cap B_R)\cup \,\rm{supp}\, \eta} (u\nabla(\eta\phi) + \eta\phi\nabla u)\,dx=\int_{\Om\cap B_R} (u\nabla(\eta\phi) + \eta\phi\nabla u)\,dx \\
&=\int_{B_R} u\eta\nabla\phi\,dx + \int_{\Om \cap B_R \cap\,\rm{supp}\, \eta} \phi
(U\nabla\eta + \eta\nabla u)\,dx.
\end{align*}
The first integral in the right-hand side is zero, $U=0$ in $B_R\setminus \Om$ and hence, $\nabla
u$ is the distributional gradient of $U$ in $B_R$, and $U \in
W^{1, \px}_{loc}(B(w,r))$. To this end, for the sake of simplicity of
notation, denote $u=U$.

Next, we show that if $\psi \in C_0^{\infty}(B(w,r))$ and $\psi \geq 0$,
then
\begin{align}\label{eq:existence_measure1}
 \int \limits_{\R^n} \langle  |\nabla u|^{p(x)-2} \nabla u, \nabla \psi \rangle dx \leq 0.
\end{align}
To prove \eqref{eq:existence_measure1}, define
\[
\psi_1:=[(\delta+\max\{u-\eps, 0\})^{\eps}-\delta^\eps]\psi.
\]
Assume that $\eps, \delta>0$ are small enough, so that $\psi_1$ is
an admissible test function for Definition \ref{px-defn}. Then we
follow the steps of Lemma 2.2 in Lundstr\"om--Nystr\"om~\cite{LN11}
for an $A$-harmonic operator $A(x, \nabla u):=|\nabla
u(x)|^{\p-2}\nabla u(x)$ to obtain that $u$ is a subsolution in
$B(w, r)$.
Hence \eqref{eq:existence_measure1} is true.

Let $K\subset B(w,r)$ be compact.
By relation between the modular and the norm \eqref{ineq:norm-mod} we have that
\begin{equation}\label{eq:existence_measure4}
\|1\|_{L^{\px}(K)}\leq \max\{r^{\frac{n}{p^+}}, r^{\frac{n}{p^-}}\}.
\end{equation}
The variable exponent H\"older inequality \eqref{ineq:Holder} together with \eqref{eq:existence_measure4} give us that
for every compact $K \subset B(w, r)$ and every $\psi \in C_0^{\infty}(K)$ it holds that
\begin{align}\label{eq:existence_measure2}
\bigg|\int \limits_{\R^n} \langle  |\nabla u|^{p(x)-2} \nabla u,
\nabla \psi \rangle dx\bigg| \leq \int \limits_{K} |\nabla
u|^{p(x)-1} |\nabla\psi| dx
\leq 2 \sup_K |\nabla\psi|\,\, \|1\|_{L^{\px}(K)}\||\nabla u|^{\px-1}\|_{L^{p'(\cdot)}(K)}\nonumber\\
\leq C \sup_K |\nabla\psi| \max\{r^{\frac{n}{p^+}}, r^{\frac{n}{p^-}}\} \max\Bigg\{ \bigg( \int \limits_{K} |\nabla u|^{p(x)} dx \bigg)^\frac{p^+-1}{p^-},
\bigg(  \int \limits_{K} |\nabla u|^{p(x)} dx \bigg)^\frac{p^--1}{p^+} \Bigg\}.
\end{align}
Take $\eta\in C_{0}^{\infty}(B(w, 2r))$ with $0\leq \eta \leq 1,\eta \equiv 1$ on
$B(w, r)$ and $|\nabla \eta|\leq \frac{C}{r}$ for some $C > 1$.
We apply Lemma \ref{energy} to get
\begin{align}\label{eq:existence_measure3}
\int_K |\nabla u|^{p(x)} dx
&\leq \int_{B(w, r)}|\nabla u|^{\p} dx
\leq \int_{B(w, 2r)} |\nabla u|^{\p} \eta^{p^+} dx \nonumber\\
&\leq c(p^+) \int_{B(w, 2r)}|u|^{\p} |\nabla \eta|^{\p} dx
\leq c(p^+)( r^{-p^+} + r^{-p^-}) \int_{B(w, 2r)}|u|^{\p} dx\nonumber\\
&\leq c(p^+, n)\,r^{n} ( r^{-p^+} + r^{-p^-}) \big(1+\sup_{\Om\cap B(w, 2r)}u\big)^{p^+}
\leq C,
\end{align}
for some constant $C$.
By \eqref{eq:existence_measure1}, \eqref{eq:existence_measure2} and \eqref{eq:existence_measure3}
it follows that $\mu$, as defined in \eqref{p(x)measure_exists:assert}, is a non-negative distribution in $B(w,r)$ and hence also a positive measure in $B(w,r)$.
Since $u$ is $\px$-harmonic in $\Omega \cap B(w,r)$ and $u \equiv 0$ in $B(w, r)\backslash \bar\Omega$,
$\mu$ has support within $\partial \Omega \cap B(w,r)$.
\end{proof}

The following theorem is the main result of this section.
In the constant exponent setting similar results are well known,
see for example Eremenko--Lewis~\cite{EL91}, Kilpel\"ainen--Zhong~\cite{KZ03} and Lundstr\"om--Nystr\"om~\cite{LN11}.
Our result in the variable exponent setting extends partially
\cite{LN11}. Indeed, by taking $p=p^+=p^-$ in
Theorem~\ref{p(x)measure} we retrieve the corresponding estimates
for $p = const$, cf. Lemma 2.7 in \cite{LN11}.

\begin{theorem}\label{p(x)measure}
Let $\Omega \subset \Rn$ be a domain having a uniformly $\px$-fat complement with constants $c_0$ and $r_0$.
Assume that $w \in\partial\Omega$, $0 < r < r_0$ and that $p$ is a log-H\"older continuous variable exponent in $\Om$ with $1<p^-\leq \px \leq p^+<n$.
Suppose that $u$ is a positive $\px$-harmonic function in $\Omega \cap B(w,r)$,
continuous on $\bar \Omega \cap B(w,r)$ with $u \equiv 0$ on $\partial \Omega \cap B(w,r)$.
Extend $u$ to $B(w, r)$ by defining $u \equiv 0$ on $B(w, r)\backslash \Omega$ and denote this extension by $u$.
Then there exist constants $C$ and $\bar c$ such that the measure $\mu$ satisfies
\begin{align*}
(i) \qquad \mu(\partial \Omega \cap B(w,\bar r))^{\frac{p^+}{p^-(p^--1)}} \leq  C  \bar r^{\frac{n-p^+}{p^--1}} \sup_{B(w, 3 \bar r)\cap \Om}u,
\end{align*}
where $\bar r = r / \bar c$.
The constant $C$ depends on $n, p^-, p^+$,
while $\bar c$ depends on $n, p^-, p^+, c_{log}, \sup_{B(w, r)\cap\Om} u$ and $c_0, r_0$.

If in addition $p^->2$ then we also have, with $\ti r = r / \ti c$,
\begin{align*}
(ii) \qquad \sup_{B(w,\ti r)\cap \Om} u \leq c \left( \tilde r^{\frac{p^+(p^- - n)}{(p^+)^2-p^-}}\,\mu (\partial \Omega \cap B(w,r))^{\frac{p^{-}}{(p^+)^2-p^-}} + \tilde r \right).
\end{align*}
The constants $c$ and $\ti c$ depend on $n, p^-, p^+, c_{log}, \sup_{B(w, r)\cap\Om} u$ and $c_0, r_0$,
while $c$ additionally depends on $c_H$.
\end{theorem}

\begin{rem}
The assumption that the complement of $\Omega$ is uniformly $\px$-fat can be replaced by a growth
estimate on the solution $u$ near $\partial \Omega$. In particular, we use the uniform $\px$-fatness only to be able to apply the H\"older continuity result (Lemma \ref{osc-lem}) giving \eqref{eq:u+r<1}.
\end{rem}

\begin{proof}
The proof relies on ideas of the constant $p$ case,
see Eremenko--Lewis~\cite[Lemma 1]{EL91} and Kilpel\"ainen--Zhong~\cite[Lemma 3.1]{KZ03}.
However, the setting of variable exponent PDEs is causing difficulties in a straightforward extension of $p=const$ arguments. Namely, the lack of homogeneity of $\px$-harmonic equation and the fact that the homogeneous Sobolev-Poincar\'e inequality~\eqref{var-exp-Sob-Poin} holds for norms but not for modular functions, require more caution and delicate approach.

We start by choosing $\bar c>6$ so large that with $\bar r = r / \bar c$ we obtain
\begin{align}\label{eq:u+r<1}
\sup_{B(w, 3\bar r)} u < 1 \qquad \textrm{and} \qquad  \bar r < 1.
\end{align}
That such $\bar c$ exists follows by H\"older continuity up to the boundary,
that is Lemma \ref{osc-lem}.
Indeed, in order to prove \eqref{eq:u+r<1}, put $\rho = 3\bar r = 3r / \bar c$ in Lemma \ref{osc-lem} to obtain
\begin{equation*}
\sup_{B(w,3\bar r)} u=\sup_{B(w, 3\bar r)\cap \Om} u \leq c
\bigg(\frac{3}{\bar c}\bigg)^{\beta} \big(\sup_{B(w, r)\cap \Om} u +
r\big) \leq 1
\end{equation*}
if $\bar c$ is large enough.
The constant $\bar c$ depends on $n, p^-, p^+, c_{log}, \sup_{B(w, r)\cap\Om} u$ and $c_0, r_0$.

We now prove the upper bound of the measure in Theorem
\ref{p(x)measure}. To simplify the notation we define
$\Delta(w,r):= \partial \Omega \cap B(w,r)$. Let $\theta \in
C_{0}^{\infty}(B(w, 2 \bar r))$ be such that $0\leq \theta\leq 1$,
$\theta\equiv 1$ on $B(w, \bar r)$ and $|\nabla \theta|\leq
\frac{C}{ \bar r}$ for some $C>1$. Using H\"older's
inequality~\eqref{ineq:Holder} and estimate
\eqref{eq:existence_measure2} we see that
\begin{align}\label{eq:rhs_energy1}
\mu(\Delta(w, \bar r)) &\leq  \int \limits_{\R^n}  \theta d \mu \leq
\int \limits_{\R^n} |\nabla  \theta| |\nabla u|^{\p-1} dx
\leq C  {\bar r}^{\frac{n}{p^+}-1}\| |\nabla u|^{\px-1} \|_{L^{p'(\cdot)}(B(w,2 \bar r))}  \nonumber\\
&\leq C {\bar r}^{\frac{n}{p^+}-1} \max\Bigg\{ \bigg( \int \limits_{B(w,2 \bar r)} |\nabla u|^{p(x)} dx \bigg)^\frac{p^+-1}{p^-},
\bigg(  \int \limits_{B(w,2 \bar r)} |\nabla u|^{p(x)} dx \bigg)^\frac{p^--1}{p^+} \Bigg\}.
\end{align}
Now let $\eta \in C_0^{\infty}(B(w,3 \bar r))$ with $0 \leq \eta \leq
1$, $\eta \equiv 1$ on $B(w,2\bar r)$ and $|\nabla \eta|\leq
\frac{C}{\bar r}$ for some $C > 1$. We apply Lemma~\ref{energy} and
\eqref{eq:u+r<1} to obtain
\begin{align}\label{eq:rhs_energy2}
\int_{B(w,2 \bar r)} |\nabla u|^{p(x)} dx
& \leq \int_{B(w,3 \bar r)}|\nabla u|^{\p} \eta^{p^+} dx \leq c(p^+)\, \int_{B(w,3 \bar r)}|u|^{\p} |\nabla \eta|^{\p} dx \nonumber\\
& \leq \frac{C(p^+)}{ {\bar r}^{p^+}}\, \int_{B(w, 3 \bar r)}|u|^{\p} dx
\leq C(p^+, n) \, {\bar r}^{n - p^+}\,\big(\sup_{B(w, 3 \bar r)}u\big)^{p^-}.
\end{align}
Using \eqref{eq:rhs_energy1}, \eqref{eq:rhs_energy2}, \eqref{eq:u+r<1} and the assumption  $p^+\leq n$ we obtain
\begin{align*}
\mu(\Delta(w, \bar r))
&\leq  C(p^-,p^+,n)\, {\bar r}^{\frac{n}{p^+}-1} \max\Bigg\{ \bigg( {\bar r}^{n - p^+}\,\big(\sup_{B(w, 3 \bar r)}u\big)^{p^-} \bigg)^\frac{p^+-1}{p^-},
\bigg( {\bar r}^{n - p^+}\,\big(\sup_{B(w, 3\bar r)}u\big)^{p^-} \bigg)^\frac{p^--1}{p^+} \Bigg\}\\
& \leq  C(p^-,p^+,n)\, {\bar r}^{\frac{(n-p^+)p^-}{p^+}} \big(\sup_{B(w, 3 \bar r)}u\big)^{ \frac{p^-}{p^+} (p^--1) }.
\end{align*}
which completes the proof for upper estimates of the measure $\mu$.\\

\bigskip

We next prove the lower bound of the measure in Theorem
\ref{p(x)measure}. To do so let $\ti r = r/\ti c$ be a radius, for
$\ti c$ to be determined later, and let $h$ be $\px$-harmonic in
$B(w, \ti r)$ with boundary values equal to $u$ on $\partial B(w,
\ti r)$. Note that by assumptions $u$ is continuous on $\bar{\Omega}
\cap {\bar{B}(w, \ti r)}$ and hence $u$ is well defined on $\partial
B(w, \ti r)$. Existence of $h$ follows from e.g. Theorem~3.6 in
Adamowicz--Bj\"orn--Bj\"orn~\cite{ABB13}. By the comparison
principle (Lemma~\ref{comparison_princ}) we see that $0 \leq u \leq
h$ in $B(w, \ti r)$. Now, by the Harnack inequality
(Lemma~\ref{harnack}) we have
\begin{align*}
\sup_{B(w, \ti r/2)} h \; \leq \; c_H \Big( \inf_{B(w, \ti r/2)} h +  \ti r \Big)
\end{align*}
and so
\begin{align}\label{eq:6infsup}
 \inf_{B(w, \ti r/2)} h\; \geq \; c_H^{-1}  \sup_{B(w,  \ti r/2)} u -  \ti r .
\end{align}
Using Lemma~\ref{holderkont} we obtain that for $t < 1/4$ we have,
\begin{align}\label{eq:6uM}
\sup_{B(w, t \ti r)} u \leq \bar{C} t^{\gamma} \Big(\sup_{B(w, \ti
r/2)} u + \ti r \Big).
\end{align}
Using (\ref{eq:6infsup}) and (\ref{eq:6uM}) we see that if $x \in
B(w, t\ti r)$ and $t$ is so small that $\bar{C}t^{\gamma} \leq 1/(2c_H)$,
then
\begin{align}\label{eq:6h-u}
h(x) - u(x) &\geq \inf_{B(w, \tilde r/2)} h   -   \sup_{B(w, t\tilde r)} u \nonumber\\
&\geq c_H^{-1} \sup_{B(w,  \ti r/2)} u -  \ti r   -   \bar{C} t^{\gamma} \Big(\sup_{B(w, \ti r/2)} u + \ti r \Big)  \nonumber\\
&\geq (2 c_H)^{-1} \sup_{B(w,  \ti r/2)} u   -  \left(1 + \bar{C} t^{\gamma}\right)\ti r\nonumber\\
&\geq \beta \sup_{B(w,  \ti r/2)} u - \ti r = \beta M(\ti r) - \ti r.
\end{align}
Hence, $M(\tilde r) = \sup_{B(w,\tilde r/2)} u$ while $\beta$ is a small constant satisfying $\beta\leq 1/(2c_H(1+\bar{C}t^\gamma))$ where $\gamma$ and $\bar{C}$ are from Lemma~\ref{holderkont}.
We note that $\beta$ depends on $n, p^-, p^+, c_{log}, c_H, \sup_{B(w, r)\cap\Om} u$ and $c_0, r_0$.
Next we note that by \eqref{eq:6h-u} a function
\begin{align}\label{def:psi}
\psi := \min_{B(w, \ti r)} \{h - u, \max \{0, \beta M(\ti r) -\ti r\} \}
\end{align}
is non-negative in $B(w, \ti r)$ and belongs to $W_0^{1, \px}(B(w, \ti r))$.
Using (\ref{eq:6h-u}) we also see that $\psi = \max \{0, \beta M(\ti r) -\ti r\}$ on $B(w, t\ti r)$.

We will now show that
\begin{align}\label{eq:upper_gradient_psi}
 \int \limits_{B(w,\ti r)} |\nabla \psi|^{p(x)} dx \leq 2^{p^+-1}\max \{0, \beta M(\ti r) -\ti r\} \mu (\Delta(w, \ti r)).
\end{align}
To do so,
let $\Gamma$ denote the set of points where $\nabla \psi$ exists and is nonzero and note that
\begin{align}\label{eq:upper_1}
 \int \limits_{B(w,\ti r)} |\nabla \psi|^{p(x)} dx \leq
 \int \limits_{\Gamma \cap B(w,\ti r)} \left(|\nabla h| + |\nabla u|\right)^{p(x) - 2} |\nabla h - \nabla u|^2 dx.
\end{align}
\noindent Moreover, for $\xi, \eta \in \R^n$,
\begin{align*}
\langle |\xi|^{p(x)-2}\xi - |\eta|^{p(x)-2}\eta, \xi-\eta \rangle
&= \frac12\left(|\xi|^{p(x)-2} + |\eta|^{p(x)-2}\right) |\xi - \eta|^2 + \frac12\left(|\xi|^{p(x)-2} - |\eta|^{p(x)-2}\right) \left( |\xi|^2 - |\eta|^2\right).
\end{align*}
Therefore and since $p^-\geq 2$ by assumption,
\begin{align*}
\langle |\xi|^{p(x)-2}\xi - |\eta|^{p(x)-2}\eta, \xi-\eta \rangle
&\geq \frac12\left(|\xi|^{p(x)-2} + |\eta|^{p(x)-2}\right) |\xi - \eta|^2 \geq \frac1{ 2^{p^+-1}}\left(|\xi| + |\eta|\right)^{p(x)-2} |\xi - \eta|^2.
\end{align*}
Upon using the last inequality and the fact that $h$ is $\px$-harmonic in $B(w, \ti r)$ and
$\psi$ is an appropriate test function for $h$ together with Lemma~\ref{p(x)measure_exists} we obtain
\begin{align*}
& \frac1{ 2^{p^+-1}}  \int  \limits_{\Gamma \cap B(w,\ti r)}  \left(|\nabla h| + |\nabla u|\right)^{p(x)-2} |\nabla h - \nabla u|^2 dx
\leq   \int \limits_{\Gamma \cap B(w,\ti r)}
\langle |\nabla h|^{p(x)-2} \nabla h - |\nabla u|^{p(x)-2} \nabla u, \nabla h - \nabla u  \rangle dx \\
&=  \int \limits_{B(w,\ti r)} \langle |\nabla h|^{p(x)-2} \nabla h, \nabla \psi \rangle dx
 -  \int \limits_{B(w,\ti r)} \langle |\nabla u|^{p(x)-2} \nabla u, \nabla \psi \rangle dx \\
&= -  \int \limits_{B(w,\ti r)} \langle |\nabla u|^{p(x)-2} \nabla u, \nabla \psi \rangle dx
 =    \int \limits_{B(w,\ti r)} \psi d\mu \leq \max \{0, \beta M(\ti r) -\ti r\} \mu(B(w,\ti r)).
\end{align*}
Since measure $\mu$ is supported on $\Delta(w,\ti r)$ we see that $\mu(B(w,\ti r))=\mu(\Delta(w,\ti r))$.
Hence, by the above inequality and \eqref{eq:upper_1} we see that \eqref{eq:upper_gradient_psi} holds true.

Next, by assuming $\ti c \geq \bar{c}$ it follows from
\eqref{eq:u+r<1} that we have $\beta M(\ti r) - \ti r < 1$. Note
that now $\ti c$ depends on $n, p^-, p^+, c_{log}, \sup_{B(w,
r)\cap\Om} u$ and $c_0, r_0$. Using this fact and the definition of
$\psi$ in \eqref{def:psi} we get that $0 \leq \psi \leq 1$ and thus
$ \int_{B(w,\tilde r)} |\psi|^{p(x)} dx\leq \omega_n \tilde
r^n\leq  \omega_n.$  The classical formula for the volume of the unit ball implies, that if $1\leq n \leq 12$, then $\omega_n> 1$. It follows that
\begin{equation}
 \int \limits_{B(w,\tilde r)} \bigg|\frac{\psi}{\omega_n^{1/p^-}}\bigg|^{p(x)} dx\leq
\frac{1}{\omega_n} \int \limits_{B(w,\tilde r)} |\psi|^{p(x)} dx\leq 1.\label{eq:px-meas1}
\end{equation}
If $n>12$, then $\omega_n<1$ and so in \eqref{eq:px-meas1} instead of
$\omega_n^{1/p^-}$ one has $\omega_n^{1/p^+}$. Eventually, this effects only the power of $\omega_n$ in
\eqref{ineq:px-meas} which for $\omega_n<1$ is $1-p^+/p^--p^-/p^+$ instead of $2-p^+/p^-$ but has no impact
on the other expressions in the discussion below. Therefore, we present the argument only in the case of $\omega_n>1$.

By the unit ball property \eqref{unit-ball} we get
\begin{equation}
\left\|\bigg|\frac{\psi}{\omega_n^{1/p^-}}\bigg|\right\|_{L^{\px}(B(w,\tilde r))}\leq 1
\qquad\hbox{ and }
 \int \limits_{B(w,\tilde r)} \bigg|\frac{\psi}{ \omega_n^{1/p^-}}\bigg|^{p(x)} dx\leq
 \left\|\bigg|\frac{\psi}{ \omega_n^{1/p^-}}\bigg|\right\|^{p^-}_{L^{\px}(B(w,\tilde r))}.\label{eq:px-meas2}
\end{equation}
This estimate, the definition of $\psi$ and the
Poincar\'e-Sobolev type inequality (see Theorem~8.2.4 in
Diening--Harjulehto--H\"ast\"o--R\r u\v zi\v cka~\cite{DHHR}) imply
the following
\begin{align}\label{eq:lower_gradient_psi}
( \beta M(\ti r) -\ti r )^{p+} \omega_n (t\ti r)^n  \omega_n^{-p^+/p^-}&\leq  \omega_n^{-p^+/p^-}\int \limits_{B(w,\ti r)} |\psi|^{p(x)} dx \nonumber \\
&\leq \left\|\bigg|\frac{\psi}{ \omega_n^{1/p^-}}\bigg|\right\|^{p^-}_{L^{\px}(B(w,\tilde r))}
\leq \frac{C^{p^-}_{Sob}}{ \omega_n}\,\tilde r^{p^-} \| \nabla \psi \|^{p^-}_{L^{\px}(B(w,\tilde r))},
\end{align}
where $C_{Sob}$ depends on $n$ and $c_{log}$.
In order to pass from the norm of the gradient to its modular we use similar approach as in \eqref{eq:px-meas1}
and \eqref{eq:px-meas2}.
For the sake of brevity and clarity of the presentation we will skip some of the tedious computations.

Without the loss of generality we may assume that $\mu
(\Delta(w, \ti r))\leq 2^{1-p^+}$. Indeed, this can be obtained by
using the upper bound of $\mu (\Delta(w, \ti r))$ proved above ($(i)$ in Theorem \ref{p(x)measure})
together with \eqref{eq:u+r<1} and by decreasing $\ti r$ if necessary.
Note that $\tilde c$ depends on $n, p^-, p^+, c_{log}, \sup_{B(w, r)\cap\Om} u$ and $c_0, r_0$.
Then by \eqref{eq:upper_gradient_psi} we have that the modular function of
$\nabla \psi$ does not exceed value one and thus, by \eqref{ineq:norm-mod}
$$
\|\nabla \psi\|^{p^+}_{L^{\px}(B(w,\ti r))}\leq  \int \limits_{B(w,\ti r)} |\nabla \psi|^{\p} dx.
$$
We continue estimation in \eqref{eq:lower_gradient_psi}. Using the above we arrive at the following inequality,
\begin{align*}
 (\beta M(\ti r) -\ti r )^{p+} \omega_n (t\ti r)^n \omega_n^{-p^+/p^-} &\leq \frac{C^{p^-}_{Sob}}{\omega_n} \tilde r^{p^-} \bigg(\int \limits_{B(w,\ti r)} |\nabla \psi|^{\p} dx\bigg)^{\frac{p^-}{p^+}}.
\end{align*}
Hence, upon using \eqref{eq:upper_gradient_psi} and including $\omega_n^{2-p^+/p^-}$ into the constant on the right-hand side of the above inequality, we get:
\begin{align}
 (\beta M(\ti r) -\ti r )^{p^+-\frac{p^-}{p^+}} t^n \leq C \tilde r^{p^--n}\,\left(\mu (B(w, \ti r))\right)^{\frac{p^-}{p^+}},
 \label{ineq:px-meas}
\end{align}
for some $C$ depending on $n, p^-, p^+$ and $c_{log}$.
Recall that according to discussion following \eqref{eq:6uM} we have that $\bar{C}t^{\gamma} \leq 1/(2c_H)$.
Choose $t$ such that $\bar{C}t^\gamma=1/(4c_H)$.
Then \eqref{ineq:px-meas} becomes
\begin{equation*}
( \beta M(\ti r) - \ti r)^{\frac{(p^+)^2}{p^-}-1}\leq C \tilde r^{p^+-\frac{p^+}{p^-}n}\,\mu (B(w, \ti r)),
\end{equation*}
for $C$ depending on $n, p^-, p^+, c_{log}, c_H, \sup_{B(w, r)\cap\Om} u$ and $c_0, r_0$.
Thus we finally conclude
\begin{equation*}
\sup_{B(w,\ti r)} u \leq C \left( \tilde r^{\frac{p^+(p^- - n)}{(p^+)^2-p^-}}\,\mu (B(w, \ti r))^{\frac{p^{-}}{(p^+)^2-p^-}} + \tilde r \right),
\end{equation*}
for some $C$ as above. Thus, the proof of Theorem~\ref{p(x)measure} is completed.
\end{proof}
Using Theorem \ref{Carleson} and Theorem \ref{p(x)measure} we obtain the following
weak doubling property of the $\px$-harmonic measure.
\begin{cor}\label{cor:doubling}
Assume that $\Omega\subset\R^{n}$ is an NTA domain with constants $M_\Om$ and $r_\Om$,
$w\in\partial\Omega$, $0 < r < r_\Om$ and let
$\px$ be a log-H\"older continuous variable exponent in $\Om$ with $2<p^-\leq \px \leq p^+<n$.
Suppose that $u$ is a positive $\px$-harmonic function in $\Omega \cap B(w,r)$,
continuous on $\bar \Omega \cap B(w,r)$ with $u \equiv 0$ on $\partial \Omega \cap B(w,r)$.
Extend $u$ to $B(w, r)$ by defining $u \equiv 0$ on $B(w, r)\backslash \Omega$ and denote this extension by $u$.
Then the measure $\mu$ satisfies the following doubling property:
\begin{align*}
&\mu(\partial\Om\cap B(w, 2 s))^{\frac{p^+}{p^-(p^--1)}} \leq  c s^{\alpha}  \left(\,\mu (\partial\Om\cap B(w,  s))^{\frac{p^{-}}{(p^+)^2-p^-}} + s^{\beta} \right),
\end{align*}
where $s = r / c$ and the constant $c$ depends on $n, p^-, p^+, c_{log}$, $c_H, \sup_{B(w, r)\cap\Om} u, M_\Om$ and $r_\Om$.
The exponents $\alpha = \alpha(n, p^+,p^-)$ and $\beta = \beta(n, p^+,p^-)$ are given by
\begin{align*}
&\alpha=\frac{(p^+-p^-)(p^+(n-p^+-p^-)+n)}{(p^--1)((p^+)^2-p^-)}\quad \textrm{and}\quad\beta=\frac{(p^+)^2-p^--p^+(p^--n)}{(p^+)^2-p^-}.
\end{align*}
In particular, for $p=p^+=p^-$ we get $\alpha=0$ and the term $s^{\beta}$ goes away as well.
Hence we retrieve the well known doubling property of $p$-harmonic measure when $p$ is constant.
\end{cor}

\begin{proof}
Let $c$ be so large that $2s \leq \ti r$ where $\ti r$ is as in Theorem \ref{p(x)measure}.
Then
\begin{align*}
\mu(\partial\Om\cap B(w, 2 s))^{\frac{p^+}{p^-(p^--1)}} &\leq  c s^{\frac{n-p^+}{p^--1}} \sup_{B(w, 6s)\cap \Om}u.
\end{align*}
By the variable exponent Carleson estimate (Theorem \ref{Carleson}) and the Harnack inequality (Lemma~\ref{harnack}) we have
\begin{align*}
\sup_{B(w, 6s)\cap \Om}u
\leq  c \left( u(a_{6s}(w)) + s \right)
\leq  c \left( u(a_{s}(w)) + s \right)
\leq  c \big( \sup_{B(w, s)\cap \Om}u + s\big).
\end{align*}
The result now follows by applying the lower bound of the $\px$-harmonic measure in Theorem \ref{p(x)measure} and by
simplification of the arising formula.

\end{proof}

\end{document}